\RequirePackage[l2tabu, orthodox]{nag}
\RequirePackage{fixltx2e}

\documentclass[11pt]{amsart}
\usepackage{stmaryrd}
\usepackage{amsmath}
\usepackage{amsfonts}
\usepackage{amsthm}
\usepackage{mathtools}
\usepackage{txfonts}
\usepackage{hyperref}
\usepackage{graphicx,amssymb,amsfonts,amsmath,amscd}
\usepackage{textcomp}
\usepackage[all,ps,cmtip]{xy}
\usepackage{amscd}
\usepackage{xspace}
\usepackage{mathrsfs}
\usepackage{ dsfont }
\usepackage[varg]{pxfonts}

\usepackage{enumerate}
\usepackage{xcolor}
\usepackage{tikz}\usetikzlibrary{decorations.markings,matrix,arrows}
\usepackage{extarrows}

\newtheorem{thm}{Theorem}[section]
\newtheorem{prop}[thm]{Proposition}
\newtheorem{lemma}[thm]{Lemma}
\newtheorem{cor}[thm]{Corollary}
\newtheorem{conj}[thm]{Conjecture}

\newtheorem{def-prop}[thm]{Definition-Proposition}
\newtheorem{prop-def}[thm]{Proposition-Definition}

\theoremstyle{definition} % upshaped
\newtheorem{defi}[thm]{Definition}
\newtheorem{rmk}[thm]{Remark}

\newcommand{\C}{{\bf C}}

\newcommand{\N}{{\bf N}}
\newcommand{\Q}{{\bf Q}}
\newcommand{\rC}{{\mathcal C}}
\newcommand{\rD}{{\mathcal D}}

\newcommand{\rF}{{\mathcal F}}

\newcommand{\rI}{{\mathcal I}}

\newcommand{\rM}{{\mathcal M}}
\newcommand{\rO}{{\mathcal O}}
\newcommand{\rP}{{\mathcal P}}

\newcommand{\rS}{{\mathcal S}}
\newcommand{\rT}{{\mathcal T}}
\newcommand{\rU}{{\mathcal U}}
\newcommand{\rV}{{\mathcal V}}
\newcommand{\rW}{{\mathcal W}}
\newcommand{\rX}{{\mathcal X}}
\newcommand{\rY}{{\mathcal Y}}
\newcommand{\rZ}{{\mathcal Z}}

\newcommand{\ie}{\textit {i.e.}}
\newcommand{\cf}{\textit {cf.}~}
\newcommand{\vs}{\textit {vs.}~}
\newcommand{\loccit}{\textit {loc.cit.}~}

\newcommand{\resp}{\textit {resp.}~}

\newcommand{\Bl}{\mathop{\rm Bl}\nolimits} %Blow-up
\newcommand{\CH}{\mathop{\rm CH}\nolimits} % Chow groups
\newcommand{\cl}{\mathop{\rm cl}\nolimits} %cycle class map
\newcommand{\codim}{\mathop{\rm codim}\nolimits} %codimension
\newcommand{\dual}{\mathop{^\vee}\nolimits} % dual
\newcommand{\Gr}{\mathop{\rm Gr}\nolimits} % Grassmannian or Grading
\newcommand{\LGr}{\mathop{\rm LGr}\nolimits} % symplectic/lagrangian
\newcommand{\OGr}{\mathop{\rm OGr}\nolimits} % orthogonal Grassmannian
\newcommand{\Hilb}{\mathop{\rm Hilb}\nolimits}

\newcommand{\id}{\mathop{\rm id}\nolimits} %identity
\newcommand{\im}{\mathop{\rm Im}\nolimits} % image, imaginary part is \Im
\newcommand{\Ker}{\mathop{\rm Ker}\nolimits} % kernel
\renewcommand{\P}{\mathop{\bf P}\nolimits} % projective space, projectivization
\newcommand{\Pic}{\mathop{\rm Pic}\nolimits} % Picard group
\newcommand{\pr}{\mathop{\rm pr}\nolimits} % projection
\newcommand{\rank}{\mathop{\rm rank}\nolimits}
\newcommand{\Spec}{\mathop{\rm Spec}\nolimits}
\newcommand{\Sym}{\mathop{\rm Sym}\nolimits} % Symmetric products
\newcommand{\vide}{{\varnothing}} % empty set
\renewcommand{\bar}{\overline}
\newcommand{\inj}{\hookrightarrow}
\newcommand{\surj}{\twoheadrightarrow}

\newcommand{\cart}{\ar@{}[dr]|\square} % cartesian diagrams, write it after the
\renewcommand{\dual}{^{\vee}} % dual
\newcommand{\isom}{\simeq} %isomorphism
\renewcommand{\tilde}{\widetilde}

\newcommand{\deff}{\mathrel{:=}}

\renewcommand{\1}{\mathop{\mathds{1}}\nolimits} %trivial Chow motive
\begin{document}
	
	\title[The generalized Franchetta conjecture for some hyper-K\"ahler
	varieties]{The generalized Franchetta conjecture \linebreak 
		for some hyper-K\"ahler varieties}
	\author{Lie Fu}
	\address{Institut Camille Jordan, Universit\'e Claude Bernard Lyon 1, France}
	\email{fu@math.univ-lyon1.fr}
	
	\author{Robert Laterveer}
	\address{Institut de Recherche Math\'ematique Avanc\'ee, CNRS, Universit\'e de
		Strasbourg, France}
	\email{robert.laterveer@math.unistra.fr}
	
	\author{Charles Vial}
	\address{Universit\"at Bielefeld, Germany}
	\email{vial@math.uni-bielefeld.de}

	\makeatletter
	\let\@wraptoccontribs\wraptoccontribs
	\makeatother
	
	\contrib[\qquad\qquad\qquad\qquad\qquad\qquad\lowercase{with an appendix joint
		with}]{Mingmin Shen}
	\address{KdV Institute for Mathematics, University of Amsterdam, Netherlands}
	\email{M.Shen@uva.nl}

	\thanks{Lie Fu is supported by the Agence Nationale de la Recherche through
		ECOVA (ANR-15- CE40-0002), HodgeFun (ANR-16-CE40-0011), LABEX MILYON
		(ANR-10-LABX-0070) of Universit\'e de Lyon, within the program
		\emph{Investissements d'Avenir} (ANR-11-IDEX-0007), and \emph{Projet
			Inter-Laboratoire} 2017 by F\'ed\'eration de Recherche en Math\'ematiques
		Rh\^one-Alpes/Auvergne CNRS 3490.}
	\begin{abstract}
		The generalized Franchetta conjecture for hyper-K\"ahler varieties predicts that
		an algebraic cycle on the universal family of certain polarized hyper-K\"ahler
		varieties is fiberwise rationally equivalent to zero if and only if it vanishes
		in cohomology fiberwise. 
		We establish Franchetta-type results for certain low (Hilbert) powers of low
		degree K3 surfaces, 
		for the Beauville--Donagi family of Fano varieties of lines on cubic fourfolds
		and its relative square, and for $0$-cycles and codimension-$2$ cycles for the
		Lehn--Lehn--Sorger--van Straten family of hyper-K\"ahler eightfolds. We also
		draw many consequences in the direction of the Beauville--Voisin conjecture as
		well as Voisin's refinement involving coisotropic subvarieties. In the appendix,
		we establish a new relation among tautological  cycles on the square of the Fano
		variety of lines of a smooth cubic fourfold and provide some applications.
	\end{abstract}
	\maketitle
	\setcounter{tocdepth}{1}
	\tableofcontents
	
	\newpage
	\section{Introduction}\label{sect:intro}
	The original Franchetta conjecture \cite{MR0069538} (proven in \cite{MR700769},
	see also \cite{MR870734} and \cite{MR895568}) states the following\,:
	
	\begin{thm}[\cite{MR0069538}, \cite{MR700769}, \cite{MR870734}, \cite{MR895568}]
		For an integer $g\geq 2$, let $\rM_{g}$ be the moduli stack of smooth projective
		curves of genus $g$, and let $\rC\to \rM_{g}$ be the universal curve.%Denote by
		Then for any line bundle $L$ on $\rC$ and any closed point $b\in \rM_{g}$, the
		restriction of $L$ to the fiber $C_{b}$ is a multiple of the canonical bundle of
		$C_{b}$.
	\end{thm}
	
	In the case of the universal family of K3 surfaces $\rS\to \rF_{g}$, where
	$\rF_{g}$ is the moduli stack of polarized K3 surfaces of genus $g$,
	O'Grady proposed in \cite{MR3115830} the following analogue of the Franchetta
	conjecture. Recall that the Beauville--Voisin class (\cite{BVK3}) of a
	projective K3 surface $S$ is the degree-$1$ $0$-cycle class $\mathfrak{o}_{S}$
	with support any closed point lying on a rational curve of the K3 surface. It
	enjoys the property that the intersection of any two divisors, as well as the
	second (Chow-theoretic) Chern class of $S$, are multiples of $\mathfrak{o}_{S}$.
	In the sequel, the Chow groups of stacks are the ones defined in Vistoli
	\cite{Vistoli} (see also Kresch \cite{Kresch}) and are always considered with
	rational coefficients. 
	
	\begin{conj}[O'Grady \cite{MR3115830}]\label{conj:FranchettaK3}
		Notation is as above. Then for any algebraic cycle $z\in \CH^{2}(\rS)$ and any
		point $b\in \rF_{g}$, the restriction of $z$ to the fiber K3 surface $S_{b}$ is
		a multiple of the Beauville--Voisin class of $S_{b}$. 
	\end{conj}
	
	Using Mukai models, Conjecture \ref{conj:FranchettaK3} is verified in
	\cite{FranchettaK3} for K3 surfaces of genus $g\leq 10$ and $g=12, 13, 16, 18,
	20$. Otherwise, Conjecture \ref{conj:FranchettaK3} is still wide open.
	
	The main goal of the paper is to investigate the following higher-dimensional
	analogue of O'Grady's Conjecture \ref{conj:FranchettaK3} concerning projective 
	hyper-K\"ahler varieties. Recall that a smooth projective variety is called
	\emph{hyper-K\"ahler} or \emph{irreducible holomorphic symplectic}, if it is
	simply connected and $H^{2,0}$ is generated by a nowhere degenerate holomorphic
	$2$-form.
	
	\begin{conj}[Generalized Franchetta conjecture, \cf
		\cite{BergeronLi}]\label{conj:FranchettaHK}
		Let $\rF$ be the moduli stack of a locally complete family of polarized
		hyper-K\"ahler varieties, 
		and let $\rX{}\to \rF{}$ be the universal family. For any $z\in
		\CH^\ast(\rX)_{\Q}$, if its restriction to a very general fiber is homologically
		trivial then its restriction to any fiber is (rationally equivalent to) zero. 
	\end{conj}
	
	Here, $\rF$ and $\rX$ are assumed to exist in the category of smooth
	Deligne--Mumford stacks. If one prefers to avoid stacks, one can add some level
	structure and obtain a universal family in the category of quasi-projective
	varieties, \emph{cf.} \cite[Section 3.4]{BergeronLi}.

	We note that  a cycle is homologically trivial when restricted to a very general
	fiber if and only if it is homologically trivial when restricted to any fiber.
	Given any smooth family of projective varieties $\rX \to \rF$ with $\rF$ smooth,
	we will say that $\rX \to \rF$ satisfies the \emph{Franchetta property} if  for
	any $z\in \CH^\ast(\rX)_{\Q}$ which is fiberwise homologically trivial, its
	restriction to any fiber is (rationally equivalent to) zero.
	
	Although it would seem optimistic\footnote{When $g\geq 4$, the relative square
		of the universal curve of genus $g$ does not satisfy the Franchetta property
		because the degree-0 0-cycle $p_1^*K_C \cdot p_2^* K_C -\deg(K_C)p_1^*K_C\cdot
		\Delta_C$ is not rationally trivial for $C$ very general of genus $g\geq 4$\,;
		see \cite{gg}.} 
	that Conjecture \ref{conj:FranchettaHK} could hold more generally for
	self-products of hyper-K\"ahler varieties -- \ie, $\rX \times_{\rF}
	\cdots\times_{\rF} \rX \to \rF$ satisfies the Franchetta property in the sense
	above -- we may nevertheless ask, given a locally complete family $\rX \to \rF$
	of polarized hyper-K\"ahler varieties, for which integers $n$ does $\rX^{n/
		\rF}$ satisfy the Franchetta property. We provide some results in that direction
	in Theorems \ref{thm:mainHilb2}, \ref{thm:mainHilbmore}, \ref{thm:mainBD2} and
	\ref{thm:mainLLSvS} below.
	
	Recently, Bergeron and Li \cite[Theorem 8.1.1]{BergeronLi} have proven the
	cohomological version of the generalized Franchetta conjecture
	\ref{conj:FranchettaHK} for relative $0$-cycles when the second Betti number is
	sufficiently large, which is an important support in favor of the conjecture, at
	least for 0-cycles.
	
	Let us also mention that Conjecture \ref{conj:FranchettaHK} is closely related
	to the so-called Beauville--Voisin conjecture and its refinement (see
	Conjectures \ref{conj:BV} and \ref{conj:BV+}). On the one hand, the proof of
	some of our main results actually uses some known cases of the Beauville--Voisin
	conjecture (especially \cite{MR2435839})\,; on the other hand, the generalized
	Franchetta conjecture implies the part of the Beauville--Voisin conjecture
	involving only Chern classes and the polarization, see Proposition
	\ref{prop:ChernBV}.
	
	We outline the main results of the paper, which provide more evidence for the
	generalized Franchetta conjecture.
	
	\subsection{Powers and Hilbert powers of some K3 surfaces}
	We can establish Franchetta-type results for the relative squares and cubes, as
	well as the relative Hilbert squares and Hilbert cubes, of the universal family
	of K3 surfaces which are complete intersections in projective spaces.
	\begin{thm}\label{thm:mainHilb2}
		Let $\rM$ be the moduli stack of smooth K3 surfaces of genus $g=3, 4$ or $5$,
		and let $\rS\to \rM$ be the universal family. Let $\rX$ be $\rS\times_{\rM}
		\rS$, $\rS\times_{\rM}\rS\times_{\rM} \rS$, $\Hilb^{2}_{\rM}\rS$,
		$\rS\times_{\rM}\Hilb^{2}_{\rM}\rS$ or $\Hilb^{3}_{\rM}\rS$. For any cycle $z\in
		\CH^\ast(\rX)_{\Q}$ and any point $b\in \rM$, the restriction of $z$ to the
		fiber $X_{b}$ is zero if and only if it is numerically trivial.
	\end{thm}
	The proof will be given in \S \ref{sect:Hilb2K3} for squares and Hilbert squares
	and in \S\ref{subsect:Hilb3K3} for the other cases. We note that, thanks to the
	result of de Cataldo and Migliorini \cite{MR1919155}, the crucial cases are the
	self-products $\rS\times_{\rM} \rS$, $\rS\times_{\rM}\rS\times_{\rM} \rS$. 

	By pushing our techniques further (\cf  \S\ref{subsect:SPB}), we can also treat
	some other cases of (Hilbert) powers of K3 surfaces\,:
	\begin{thm}\label{thm:mainHilbmore}
		The following families satisfy the Franchetta property\,:
		\begin{enumerate}[$(i)$]
			\item $\rS\times_{\rM}\rS$, $\Hilb^{2}_{\rM}\rS$,
			$\rS\times_{\rM}\rS\times_{\rM}\rS$, $\rS\times_{\rM}\Hilb^{2}_{\rM}\rS$ and
			$\Hilb^{3}_{\rM}\rS$, where $\rS\to \rM$ is the universal family of smooth K3
			surfaces of genus 2 (double planes).
			\item
			$\Hilb^{r_{1}}_{\rM}\rS\times_{\rM}\cdots\times_{\rM}\Hilb^{r_{m}}_{\rM}\rS$,
			where $\rS\to \rM$ is the universal family of smooth quartic K3 surfaces and
			$r_{1}+\cdots+r_{m}\leq 5$.
			\item The relative square and relative Hilbert square of the universal family of
			K3 surfaces of genera $6, 7, 8, 9, 10, 12$.
		\end{enumerate} 
	\end{thm}
	The proof will be given in \S \ref{subsect:beyond}, where these results are just
	special cases of the more general but more technical Theorem \ref{thm:beyond}.
	See also Remark \ref{rmk:limit} which explains that the ranges in
	Theorems~\ref{thm:mainHilb2} and \ref{thm:mainHilbmore} above are, at least most
	of them, already at the limit of our method.
	
	As immediate consequences, we obtain some partial confirmation of Voisin's
	refinement of the Beauville--Voisin conjecture involving coisotropic
	subvarieties (Conjecture \ref{conj:BV+})\,:
	
	\begin{cor}\label{cor:BV+} Let $S$ be a general K3 surface of genus $g\leq 10$
		or 12, and let $X$ be the Hilbert square $X=\Hilb^{2}(S)$. Let $R^\ast(X)\subset
		\CH^\ast(X)_{\Q}$ denote the 
		$\Q$-subalgebra generated by the polarization class $h$, the Chern classes
		$c_{i}$, and the Lagrangian surface $T\subset X$ constructed in
		\cite[Proposition 4]{IM}. Then $R^\ast(X)$ injects into cohomology by the cycle
		class map.
	\end{cor}
	
	\begin{cor}\label{cor:BV++} Let $S\subset\P^3$ be a quartic K3 surface, and let
		$X=\Hilb^{5}S$, $\Hilb^{2}S\times \Hilb^{2}S\times S$, $\Hilb^{2}S\times S^{3}$
		or $\Hilb^{2}S\times \Hilb^{3}S$. Let $R^\ast(X)\subset \CH^\ast(X)_{\Q}$ denote
		the 
		$\Q$-subalgebra generated by the polarization class $h$, the Chern classes
		$c_{i}$, the coisotropic subvarieties $E_\mu$ of \cite[4.1 item 1)]{MR3524175},
		the Lagrangian surface $T\subset \Hilb^{2}S$ constructed in \cite[Proposition
		4]{IM}, and the surface of bitangents $U\subset \Hilb^{2}S$. 
		Then $R^\ast(X)$ injects into cohomology by the cycle class map.
	\end{cor}
	
	These two corollaries are proven in \S\ref{subsect:beyond} and also partially
	extended to products of Hilbert schemes in Corollary \ref{cor:BV+r}. A similar
	application to a $19$-dimensional family of double EPW sextics is given in \S
	\ref{epw}.
	
	Another consequence, whose proof as well as the background is in \S
	\ref{subsect:AppBloch}, concerns the Bloch conjecture for the anti-symplectic
	involution on Hilbert squares of quartic surfaces constructed by Beauville
	\cite{Kata}\,:
	
	\begin{cor}\label{cor:blochkata} Let $X=\Hilb^{2}S$ be the Hilbert square of a
		quartic K3 surface $S$, and let $\iota\colon X\to X$ be the anti-symplectic
		involution of Beauville \cite{Kata}. Then
		\[ \begin{split}  \iota^\ast=-\id\colon\ \ \ \CH^i(X)_{(2)}\ &\to\
		\CH^i(X)_{(2)}\ \ \ (i=2,4)\ ,\\
		\iota^\ast=\id\colon\ \ \ \CH^4(X)_{(j)}\ &\to\
		\CH^4(X)_{(j)}\ \ \ (j=0,4)\ .\\
		\end{split}\]
		(Here, the notation $\CH^\ast(X)_{(\ast)}$ refers to the Fourier
		decomposition of $\CH^\ast(X)_{\Q}$ constructed by Shen--Vial \cite{SV}.)
	\end{cor}

	\subsection{The Beauville--Donagi family}
	
	For the universal family of Fano varieties of lines of cubic fourfolds, which
	form a locally complete family of projective hyper-K\"ahler fourfolds of
	K3$^{[2]}$-type (\cite{MR818549}), we have the following slightly stronger
	result than predicted by Conjecture \ref{conj:FranchettaHK}\,:
	\begin{thm}\label{thm:mainBD}
		Let $\rC$ be the moduli stack of smooth cubic fourfolds, $\rX\to \rC$ the
		universal family and $\rF\to \rC$ be the universal family of Fano varieties of
		lines of the fibers of  $\rX/\rC$. 
		Then for any $i\in \N$, any $z\in \CH^{i}(\rF)_{\Q}$ and any $b\in \rC$, the
		restriction of $z$ to the fiber $F_{b}$ is numerically trivial if and only if it
		is (rationally equivalent to) zero.
		\footnote{In fact, we show that the restriction of $\CH^*(\rF)_{\Q}$ to
			$\CH^*(F_b)_{\Q}$ is the \emph{tautological} subring, which is defined as the
			$\Q$-subalgebra generated by the Pl\"ucker polarization of $F_b$ and by the
			Chern classes of $F_b$, see Remark \ref{rmk:TautologicalRing}.} 
	\end{thm}

	In order to study the next case (Theorem \ref{thm:mainLLSvS}), we also prove the
	following analogous result on the relative square of the universal family of
	Fano varieties of lines\,:
	\begin{thm}\label{thm:mainBD2}
		Notation is as in Theorem \ref{thm:mainBD}. Then for $z\in
		\CH^{i}(\rF\times_{\rC}\rF)_{\Q}$ and any $b\in \rC$, the restriction of $z$ to
		the fiber $F_{b}\times F_{b}$ is numerically trivial if and only if it is
		(rationally equivalent to) zero.\footnote{We actually show that the restriction
			of $\CH^{\ast}(\rF\times_{\rC}\rF)_{\Q}$ to $\CH^\ast(F_{b}\times F_{b})_{\Q}$
			is the \emph{tautological} subring, which is defined as the $\Q$-subalgebra
			generated by the tautological subrings of the two factors together with the
			classes of the diagonal and the incidence subvariety\,; see Proposition
			\ref{prop:AllTaut}.}
	\end{thm}
	
	The proof of Theorem \ref{thm:mainBD} (resp. Theorem \ref{thm:mainBD2}) consists
	of two steps. First we show that cycles that belong to the image of the 
	restriction map  $\CH^{i}(\rF)_{\Q} \to \CH^i(F_b)_\Q$
	(\emph{resp.}~$\CH^{i}(\rF\times_{\rC}\rF)_{\Q} \to \CH^i(F_b\times F_b)_\Q$)
	are tautological in the sense of Remark \ref{rmk:TautologicalRing}
	(\emph{resp.}~Definition \ref{def:tautFxF}). Second we show that relations among
	tautological cycles modulo numerical equivalence in fact hold modulo rational
	equivalence.  More precisely, we determine completely in terms of generators and
	relations the rings of tautological cycles for $F_b$ and $F_b\times F_b$. In the
	case of $F_b\times F_b$, all relations but one had been established in
	\cite{MR2435839} and \cite{SV}. The remaining relation is established in a joint
	appendix with Mingmin Shen, where we also draw some consequences concerning the
	multiplicative properties of the Chow motive of $F_b$.

	\subsection{The Lehn--Lehn--Sorger--van Straten family}
	
	Similarly to the Fano varieties of lines of cubic fourfolds,
	Lehn--Lehn--Sorger--van Straten (LLSvS) consider in \cite{LLSvS} the twisted
	cubic curves on a cubic fourfold not containing a plane and show that the base
	of the maximal rationally connected (MRC) quotient of the moduli space of such
	curves is a hyper-K\"ahler eightfold. Later Addington and M. Lehn show in
	\cite{AddingtonLehn} that this hyper-K\"ahler eightfold is of
	K3$^{[4]}$-deformation type (\emph{cf.} also \cite{CLehn}). For the universal
	family of LLSvS hyper-K\"ahler eightfolds, we have the following result, which
	confirms the $0$-cycle and codimension-$2$ cases of the generalized Franchetta
	conjecture.
	
	\begin{thm}\label{thm:mainLLSvS}
		Let $\rC^{\circ}$ be the moduli stack of smooth cubic fourfolds not containing a
		plane and let $\rZ\to \rC^{\circ}$ be the universal family of LLSvS
		hyper-K\"ahler eightfolds (\cite{LLSvS}). Then
		\begin{enumerate}[(i)]
			\item  for any $b\in \rC^{\circ }$ and for any $\gamma\in \CH^{8}(\rZ)$ which is
			fiber-wise of degree 0, the restriction of $\gamma$ to the fiber $Z_{b}$ is
			(rationally equivalent to) zero.
			\item  for any $b\in \rC^{\circ }$ and for any $\gamma\in \CH^{2}(\rZ)_{\Q}$,
			its restriction to the fiber $Z_{b}$ is zero if and only if its cohomology class
			vanishes.
		\end{enumerate}
	\end{thm}
	
	As a consequence, we deduce a part of the Beauville--Voisin Conjecture
	\ref{conj:BV} as well as the refined Conjecture \ref{conj:BV+} for LLSvS
	eightfolds\,:
	\begin{cor}\label{cor:BVChern}
		Given any smooth cubic fourfold $X$ which does not contain a plane, let $Z$ be
		the LLSvS hyper-K\"ahler eightfold associated to $X$. Denote by $h$ the
		polarization class. Then the classes $$h^{8}, c_{2}h^{6}, c_{2}^{2}h^{4},
		c_{2}^{3}h^{2}, c_{2}^{4},c_{4}h^{4}, c_{2}c_{4}h^{2}, c_{2}^{2}c_{4},
		c_{6}h^{2}, c_{2}c_{6}, c_{4}^{2}, c_{8}\in \CH_{0}(Z)_{\Q}$$ are all
		proportional, where $c_{i}:=c_{i}(T_{Z})$ is the $i$-th (Chow-theoretic) Chern
		class of the tangent bundle of $Z$. We call the generator of degree 1 in this
		one-dimensional subspace the \emph{canonical 0-cycle class} or the
		\emph{Beauville--Voisin class} of $Z$, denoted by $\mathfrak{o}_{Z}$.\\
		More strongly, let $R^\ast(Z)$ be the $\Q$-subalgebra generated by the
		polarization class $h$, the Chern classes $c_{i}$ together with the following
		classes of coisotropic subvarieties of $Z$\,:
		\begin{itemize}
			\item the embedded cubic fourfold $X\subset Z$ (\cite{LLSvS})\,;
			\item the space of twisted cubics contained in a general hyperplane section of
			$X$ (\cite{ShinderSoldatenkov})\,;
			\item the coisotropic subvarieties of codimension $1,2,3,4$ constructed by
			Voisin \cite[Corollary 4.9]{MR3524175}\,;
			\item the fixed locus of the anti-symplectic involution $\iota$ of $Z$
			(\cite{LLMS})\,;
			\item the images by $\iota$ of all the above subvarieties.
		\end{itemize}
		Then $R^{8}(Z)=\Q\cdot \mathfrak{o}_{Z}$.
	\end{cor}

	\noindent\textbf{Conventions.} All algebraic varieties are over the field of
	complex numbers. We work with Chow groups with rational coefficients. For the
	$m$-th Hilbert scheme of a surface $S$, the two notations $S^{[m]}$ and
	$\Hilb^{m}(S)$ are used interchangeably and similarly for the relative
	situation. Chow groups of Deligne--Mumford stacks are the ones defined with
	rational coefficients by Vistoli \cite{Vistoli} (there is a definition with
	integer coefficients by Kresch \cite{Kresch}).

	\noindent\textbf{Acknowledgements.} The authors want to thank Nicolas Addington,
	Zhiyuan Li, Renjie Lyu, Nicolas Ressayre, Qizheng Yin for their interest and
	helpful comments and discussions. Thanks to the referee for many pertinent
	suggestions that helped improve the paper.

	\section{General remarks}
	
	\subsection{Generic fiber \vs geometric fibers}
	There is the following slightly different version of the generalized Franchetta
	conjecture for hyper-K\"ahler varieties\,:
	\begin{conj}\label{conj:FranchettaWeak}
		Let $\rF$ be the moduli stack of certain polarized hyper-K\"ahler varieties and
		let $\pi: \rX\to \rF$ be the universal family. Denote by $\rX_{\eta}$ the
		generic fiber of $\pi$, where $\eta=\Spec(\C(\rF))$. Then the group
		$\CH^{*}(\rX_{\eta})_{\hom}$ is zero.
	\end{conj}
	Here homological equivalence is with respect to some classical Weil
	cohomology\,; for instance, \'etale cohomology or de Rham cohomology. 
	
	\begin{lemma}\label{lemma:Compare}
		Conjecture \ref{conj:FranchettaHK} and Conjecture \ref{conj:FranchettaWeak} are
		equivalent.
	\end{lemma}
	\begin{proof}
		Let us start by assuming Conjecture \ref{conj:FranchettaHK}. 
		Using \cite[Lemma 2.1]{MR3158241}, the hypothesis that the restriction of $z$ to
		the geometric generic fiber is homologically trivial implies that the
		restriction of $z$ to a very general geometric fiber is also trivial. Now the
		conclusion of Conjecture \ref{conj:FranchettaHK} says that the restriction of
		$z$ to a very general geometric fiber is (rationally equivalent to) zero. By the
		standard argument of decomposition of the diagonal (\cite{MR714776},
		\cite{MR1997577}, \cite{MR3186044}), this implies the existence of a Zariski
		open dense subset $U\subset \rF$, such that $z|_{\rX_{U}}$ is zero. In
		particular, $z_{\eta}$ is rationally equivalent to zero.\\
		For the other direction, since we know that $\CH^{*}(\rX_{\eta})_{\hom}=0$, by
		restriction we can show Conjecture ~\ref{conj:FranchettaHK} for general fibers.
		Then a standard specialization argument allows us to conclude for all fibers.
	\end{proof}
	
	Thanks to Lemma \ref{lemma:Compare}, we will focus in this paper on Conjecture
	\ref{conj:FranchettaHK}.
	
	\subsection{Relation to Beauville--Voisin conjecture}
	As is mentioned in the introduction, the generalized Franchetta conjecture
	\ref{conj:FranchettaHK} is very much related to the following Beauville--Voisin
	conjecture\,:
	\begin{conj}[Beauville--Voisin \cite{MR2187148},
		\cite{MR2435839}]\label{conj:BV}
		Let $X$ be a projective hyper-K\"ahler variety. Let the \emph{Beauville--Voisin
			subring} $\langle c_{i}(X), Pic(X)\rangle$ be the $\Q$-subalgebra of
		$\CH^{*}(X)$ generated by line bundles and all (Chow-theoretic) Chern classes of
		$T_{X}$. Then the restriction of the cycle class map to the Beauville--Voisin
		subring is injective. In other words, any polynomial of line bundles and Chern
		classes of $X$ is homologically equivalent to zero if and only if it is
		rationally equivalent to zero.
	\end{conj}
	The original version due to Beauville in \cite{MR2187148}, under the name of
	\emph{weak splitting property}, contains only line bundles\,; the Chern classes
	of the tangent bundle are introduced by Voisin in \cite{MR2435839}. Some active
	progress towards this conjecture has recently been made\,: see \cite{MR2187148},
	\cite{MR2435839}, \cite{MR3356741},  \cite{MR3351754}, \cite{MR3579961},
	\cite[Theorem 1.14]{MHRCK3} for the known results and more details. More
	recently, Voisin \cite{MR3524175} proposes the following stronger version of
	Conjecture \ref{conj:BV} involving some coisotropic subvarieties.
	Recall that a subvariety is called \emph{coisotropic} if the tangent space at
	each regular point of this subvariety is a coisotropic subspace
	(\emph{i.e.}~containing its orthogonal) with respect to the holomorphic
	symplectic form. We say that a subvariety of codimension $i$ is \emph{strongly
		coisotropic} if it can be swept out by $i$-dimensional subvarieties that are
	constant cycle subvarieties of the ambient hyper-K\"ahler variety. (Naturally, a
	strongly coisotropic subvariety is coisotropic.)
	\begin{conj}[Voisin's refinement \cite{MR3524175}]\label{conj:BV+}
		Let $X$ be a projective hyper-K\"ahler variety.  Then the restriction of the
		cycle class map to the $\Q$-subalgebra of $\CH^\ast(X)$ generated by line
		bundles, Chern classes of $T_{X}$ and strongly coisotropic subvarieties, is
		injective. 
	\end{conj}

	We would like to point out that the generalized Franchetta conjecture implies
	the part of the Beauville--Voisin conjecture involving only the Chern classes of
	the tangent bundle and the polarization class. More generally it actually
	implies part of the refined Conjecture \ref{conj:BV+} once taking into account
	strongly coisotropic subvarieties which are defined universally over the moduli
	space (see Corollaries \ref{cor:BV+}, \ref{cor:BV++}, \ref{cor:doubleEPW} and
	\ref{cor:BVChern} for examples)\,:
	\begin{prop}\label{prop:ChernBV}
		Let $\rF$ be a moduli space of polarized hyper-K\"ahler varieties. If Conjecture
		\ref{conj:FranchettaHK} holds true for the universal family over $\rF$, then for
		any member $X$ of this family, the cycle class map restricted to the
		$\Q$-subalgebra generated by the polarization line bundle and the Chern classes
		of $X$, is injective.\\
		More generally,  still assuming Conjecture \ref{conj:FranchettaHK}, for any
		member $X$ of this family, the cycle class map restricted to the $\Q$-subalgebra
		generated by the algebraic cycles of $X$ that exist universally over the moduli
		space, is injective.
	\end{prop}
	\begin{proof}
		For any member $X$ and any given polynomial in the polarization line bundle and
		the Chern classes of the tangent bundle $z:=P(h, c_{i}(T_{X}))\in \CH^\ast(X)$
		such that the cohomology class of $z$ vanishes, we want to show that $z=0$.
		Consider $\gamma:= P(h, c_{i}(T_{\rX/\rF}))\in \CH^\ast(\rX)$. Clearly
		$\gamma|_{X}=z$ and hence $\gamma$ has fiber-wise vanishing cohomology class.
		Then the generalized Franchetta conjecture \ref{conj:FranchettaHK} says exactly
		that $z$ is rationally equivalent to zero. 
		The last assertion is more or less tautological.
	\end{proof}

	\subsection{Moduli space \vs parameter space}
	\begin{rmk}\label{rmk:Parameter}
		To establish the generalized Franchetta conjecture (or more generally the
		Franchetta property) in some cases, it will be convenient to work over some
		parameter space which dominates the moduli stack, instead of the moduli stack
		itself. More precisely, keep the same notation as in Conjecture
		\ref{conj:FranchettaHK} and let $B \to U$ be a surjective morphism from some
		smooth parameter space $B$ (it will often be denoted by $B^{\circ}$ in concrete
		situations) to some smooth Zariski dense open subset $U$ of the moduli stack
		$\rF$. Denote by $\rY\to B$ the pulled-back family of the universal family
		$\rX\to \rF$. Then the generalized Franchetta conjecture for $\rY\to B$ implies
		the generalized Franchetta conjecture for $\rX\to \rF$ (but not conversely).
		\begin{equation*}
		\xymatrix{
			\rY \cart\ar[d] \ar@{->>}[r] & \rX_{U}\cart \ar[d] \ar@{^{(}->}[r] &\rX\ar[d]\\
			B \ar@{->>}[r] & U \ar@{^{(}->}[r] & \rF
		}
		\end{equation*}
		Indeed, for any $z\in \CH^\ast(\rX)$, denote by $z'\in \CH^\ast(\rY)$ its
		pull-back image in $\rY$. Obviously, the hypothesis that the restriction of $z$
		to a very general fiber of $\rX/\rF$ is homologically trivial implies the same
		thing for the restriction of $z'$ to the fibers of $\rY/B$. The generalized
		Franchetta conjecture for $\rY/B$ then implies that $z'$ restricts to zero on
		each fiber of $\rY/B$. Hence so does $z$ for each fiber of $\rX_{U}\to U$. A
		specialization argument shows that the same thing holds for each fiber of
		$\rX\to \rF$.
	\end{rmk}

	\section{Fano varieties of lines on cubic fourfolds}\label{sect:BD}
	In this section, we prove Theorem \ref{thm:mainBD}, which by Remark
	\ref{rmk:Parameter} confirms the generalized Franchetta conjecture for the
	20-dimensional locally complete family of 
	polarized hyper-K\"ahler fourfolds constructed by Beauville--Donagi in
	\cite{MR818549}. The key idea of the proof is as in \cite{MR3099982} and
	\cite{FranchettaK3}\,: the universal family has very simple Chow groups. 
	
	We start by setting up some notations.
	Let $V$ be a 6-dimensional vector space and $\P^{5}=\P(V)$ be its
	projectivization. 
	The parameter space of possibly singular cubic fourfolds is given by the
	following projective space\,:
	$$B\deff \P\left(H^{0}(\P^{5}, \rO(3))\right)=\P(\Sym^{3}V^{\vee})\isom
	\P^{55}.$$
	Let $B^{\circ}\subset B$ be the open subset parameterizing smooth cubic
	fourfolds. We thus have the universal family $\rX\to B$ as well as the smooth
	family $\rX^{\circ}\to B^{\circ}$ by base-change.
	
	Let $G\deff \Gr(\P^{1}, \P^{5})\isom \Gr(2,6)$ be the Grassmannian variety
	parameterizing all projective lines in $\P^{5}$. Denote by $S$ (\resp $Q$) the
	tautological rank 2 subbundle (\resp rank 4 quotient bundle), fitting into the
	following short exact sequences of vector bundles over $G$\,:
	$$0\to S\to \rO_{G}\otimes V\to Q\to 0.$$
	Note that for any equation $f\in \Sym^{3}V^{\vee}$, the above short exact
	sequence gives a section $s_{f}$ of the vector bundle $\Sym^{3}S^{\vee}$, whose
	zero locus $(s_{f}=0)$ is exactly the Fano variety of lines of the cubic
	fourfold defined by $f$.
	
	Consider the incidence subvariety $\rF$ in $B\times G$ defined by
	$$\rF\deff \left\{([f], l) \in B\times G~~\vert~~ f|_{l}=0\right\},$$ together
	with the two natural projections\,:
	\begin{equation*}
	\xymatrix{
		&\rF\ar[dr]^{p}\ar[dl]_{\pi}&\\
		B&& G
	}
	\end{equation*}
	It is easy to see that $\pi: \rF\to B$ is the universal Fano variety of lines of
	fibers of $\rX/B$ and that $p: \rF\to G$ is a projective bundle whose fiber over
	a line $l\in G$ parameterizes all (possibly singular) cubic fourfolds containing
	$l$.
	
	As in \cite[Lemma 2.1]{FranchettaK3}, we have the following\,:
	\begin{lemma}\label{lemma:Restriction}
		For any $b\in B$, the following two images of restriction maps are the same\,:
		$$\im\left(\CH^\ast(\rF)\to \CH^\ast(F_{b})\right)=\im\left(\CH^\ast(G)\to
		\CH^\ast(F_{b})\right).$$
	\end{lemma}
	\begin{proof}
		The inclusion ``$\supseteq$'' is trivial (we have the factorization $F_{b}\inj
		\rF\to G$).\\
		Let us show the inverse inclusion. Given any cycle $z\in \CH^\ast(\rF)$, we have
		by the projective bundle formula
		$$z=\sum_{k\geq 0} p^{*}(z_{k})\cdot \xi^{k},$$
		where $z_{k}\in\CH^\ast(G)$ and $\xi= c_{1}\left(\rO_{p}(1)\right)$. As in
		\cite[Lemma 2.1]{FranchettaK3}, we easily check that $\xi$ is a linear
		combination of cycles pulled back from $B$ by $\pi$ and cycles pulled back from
		$G$ by $p$. Hence $z$ is a polynomial of cycles of the form $p^{*}(\alpha)$ and
		$\pi^{*}(\beta)$. The latter type being zero when restricted to any fiber
		$F_{b}$, the restriction of $z$ to $F_{b}$ is therefore the restriction of some
		cycle of $G$.
	\end{proof}
	
	\begin{lemma} \label{lemma:InBV}
		For any $b\in B^{\circ}$,
		$$\im\left(\CH^\ast(G)\to \CH^\ast(F_{b})\right)\subseteq\langle c_{i}(F_{b}), 
		\Pic(F_{b})\rangle,$$ where the right hand side is the \emph{Beauville--Voisin
			subring} of $\CH^\ast(F_{b})$ generated (as a $\Q$-algebra) by line bundles and
		all Chern classes of the tangent bundle of $F_{b}$.
	\end{lemma}
	\begin{proof}
		Since $\CH^\ast(G)$ is generated (as a $\Q$-algebra) by $c_{1}(S^{\vee})$ and
		$c_{2}(S^{\vee})$, it suffices to show that both of their restrictions to
		$F_{b}$ lie in the Beauville--Voisin ring. The first one being a line bundle, it
		remains to show that $c_{2}(S^{\vee}|_{F_{b}}) \in \langle c_{i}(F_{b}), 
		\Pic(F_{b})\rangle$. However, using the short exact sequence $$0\to T_{F_{b}}\to
		T_{G}|_{F_{b}}\to \Sym^{3}S^{\vee}|_{F_{b}}\to 0$$ together with the isomorphism
		$T_{G}\isom S^{\vee}\otimes Q$, one finds that 
		$$ch(T_{F_{b}})=ch(S^{\vee}|_{F_{b}})\left(6-ch(S|_{F_{b}})\right)-ch\left(\Sym^{3}S^{\vee}|_{F_{b}}\right),$$
		and hence
		$c_{2}(T_{F_{b}})=-ch_{2}(T_{F_{b}})=5c_{1}(S^{\vee}|_{F_{b}})^{2}-8c_{2}(S^{\vee}|_{F_{b}})$.
		Therefore $c_{2}(S^{\vee}|_{F_{b}})$ also belongs to the Beauville--Voisin
		ring.\footnote{The classes $c_{1}(S^{\vee}|_{F_{b}})$ and
			$c_{2}(S^{\vee}|_{F_{b}})$ are the classes that Claire Voisin calls $g$ and $c$
			respectively in \cite{MR2435839}.}
	\end{proof}
	We can now easily conclude\,:
	\begin{proof}[Proof of Theorem \ref{thm:mainBD}]
		Let $z$ be an element in $\CH^\ast(\rC)$.
		For any $b\in B^{\circ}$, thanks to Lemma \ref{lemma:Restriction}, $z|_{F_{b}}$
		is the restriction of some cycle from $G$, which must lie in the
		Beauville--Voisin ring $\langle c_{i}(F_{b}),  \Pic(F_{b})\rangle$ by Lemma
		\ref{lemma:InBV}. Now the equivalence between homological triviality and
		rational triviality of $z|_{F_{b}}$ is a consequence of Voisin's result
		\cite[Theorem 1.4(ii)]{MR2435839} saying that the cycle class map restricted to
		the Beauville--Voisin ring is injective. Finally, numerical equivalence and
		homological equivalence coincide for Fano varieties of lines of cubic fourfolds
		by \cite{MR3040747}.
	\end{proof}
	\begin{rmk}\label{rmk:TautologicalRing}
		In fact, the above proof shows that the restriction of a cycle $z\in
		\CH^\ast(\rC)$ to a fiber Fano variety of lines $F$ is in the so-called
		\emph{tautological} ring $R^{*}(F)$, which is the $\Q$-subalgebra of
		$\CH^{*}(F)$, in general smaller than the Beauville--Voisin ring, generated by
		the Pl\"ucker polarization class $g$ and the Chern classes of $F$. In
		particular, 
		\begin{itemize}
			\item $R^{1}(F)=\Q\cdot g$\,;
			\item $R^{2}(F)=\Q\cdot g^{2}\oplus \Q\cdot c_{2}$\,;
			\item $R^{3}(F)=\Q\cdot g^{3}$ (by \cite[Lemma 3.5]{MR2435839}, $gc_{2}$ and
			$g^{3}$ are proportional)\,;
			\item $R^{4}(F)=\Q\cdot \mathfrak{o}_{F}$, where $\mathfrak{o}_{F}$ is the
			canonical 0-cycle class and $c_{2}^{2}, c_{4}, g^{4}, g^{2}c_{2}$ are all
			proportional to it by \cite[Lemma 3.2]{MR2435839}. 
		\end{itemize}
	\end{rmk}

	\section{Hilbert squares of complete intersection K3
		surfaces}\label{sect:Hilb2K3}
	In this section, we prove Theorem \ref{thm:mainHilb2} for squares and Hilbert
	squares. There are three families of complete intersection K3 surfaces, namely,
	quartic surfaces in $\P^{3}$, complete intersections of quadric and cubic
	hypersurfaces in $\P^{4}$ and complete intersections of three quadric
	hypersurfaces in $\P^{5}$.
	
	Let us fix some notations. In each of the three cases\,:
	\begin{itemize}
		\item $\P:=\P^{3}, \P^{4} \text{ \resp } \P^{5}$ is the ambient projective
		space\,;
		\item $E:=\rO_{\P}(4), \rO_{\P}(2)\oplus\rO_{\P}(3), \text{ \resp }
		\rO_{\P}(2)^{\oplus3}$ is the relevant vector bundle\,;
		\item  $B:= \P H^{0}(\P, E)$ is the parameter (projective) space and
		$B^{\circ}$ is the open subset parameterizing smooth K3 surfaces.
		\item $\rS\deff\left\{(x, [s]) \in \P\times B ~~\vert~~ s(x)=0\right\}$ is the
		universal family. 
	\end{itemize}
	
	We have therefore the two natural projections, where $p$ is clearly a projective
	bundle\,;
	\begin{equation}\label{diag:K3}
	\xymatrix{
		\rS \ar[r]^{p} \ar[d]_{\pi} & \P\\
		B&
	}
	\end{equation}
	Similarly, the relative square and the open complement of the relative diagonal
	in it fit into the following diagram
	\begin{equation}\label{diag:doubleK3}
	\xymatrix{
		\rS\times_{B}\rS\backslash{\Delta_{\rS/B}} \ar[r]^{q'} \ar@{^{(}->}[d]_{j} &
		\P\times \P\backslash \Delta_{\P}\ar@{^{(}->}[d] \\
		\rS\times_{B}\rS \ar[r]^{q\deff (p,p)} \ar[d]_{\pi_{2}\deff (\pi, \pi)} &
		\P\times \P\\
		B&
	}
	\end{equation}
	Note that although $q$ itself is not a projective bundle, its restriction $q'$
	is. Let $\xi$ be the first Chern class of $\rO_{q'}(1)$. The relative diagonal
	$\Delta_{\rS/B}$ being of codimension 2, $\xi$ extends uniquely to the whole of
	$\rS\times_{B}\rS$, which we still denote by $\xi$ by abuse of notation.
	
	We can show the analogue of Lemma \ref{lemma:Restriction} in our
	case\footnote{Proposition \ref{prop:Restriction} will be generalized for the
		so-called stratified projective bundle in \S \ref{subsect:SPB}.} \,:
	\begin{prop}\label{prop:Restriction}
		For any $b\in B^{\circ}$, we have\,:
		$$\im\left(\CH^\ast(\rS\times_{B}\rS)\to \CH^\ast(S_{b}\times
		S_{b})\right)=\im\left(\CH^\ast(\P\times \P)\to \CH^\ast(S_{b}\times
		S_{b})\right)+ \Delta_{*}\im\left(\CH^\ast(\P)\to \CH^\ast(S_{b})\right),$$
		where $\Delta: S_{b}\inj S_{b}\times S_{b}$ is the diagonal embedding.
	\end{prop}
	\begin{proof}
		Notation is as in Diagrams (\ref{diag:K3}) and (\ref{diag:doubleK3}). By
		base-change, it is easy to see that the right-hand side is contained in the
		left-hand side. Concerning the inverse inclusion,
		the projective bundle formula gives, for any $z\in
		\CH^\ast(\rS\times_{B}\rS)_{}$, $$j^{*}(z)=\sum_{k\geq 0} q'^{*}(z_{k})\cdot
		\xi^{k},$$ for some cycles $z_{k}\in \CH^\ast(\P\times\P\backslash
		\Delta_{\P})$. As in Lemma \ref{lemma:Restriction}, it is easy to see that
		$\xi=j^{*}\pi_{2}^{*}(h)+q'^{*}(\alpha)$, where $h=c_{1}(\rO_{B}(1))$ and
		$\alpha\in \CH^\ast(\P\times \P\backslash\Delta_{\P})$.
		For each $k$, we denote still by $z_{k}\in \CH^\ast(\P\times \P)$ its closure
		and similarly for $\alpha$.
		Therefore, we have
		$$z-\sum_{k}q^{*}(z_{k})\cdot\left(\pi^{*}_{2}(h)-q^{*}(\alpha)\right)^{k}\in
		\Ker(j^{*}).$$
		By the localization sequence, there exists $\gamma\in \CH^\ast(\rS)$, such that 
		\begin{equation}\label{eqn:z}
		z-\sum_{k}q^{*}(z_{k})\cdot\left(\pi^{*}_{2}(h)-q^{*}(\alpha)\right)^{k}=\Delta_{*}(\gamma),
		\end{equation}
		where $\Delta: \rS\inj \rS\times_{B}\rS$ is the diagonal embedding.
		
		Since $p: \rS\to \P$ is also a projective bundle with
		$c_{1}(\rO_{p}(1))=\pi^{*}(h)$, we have
		$$\gamma=\sum_{l}p^{*}(\gamma_{l})\cdot\pi^{*}(h)^{l},$$ for some $\gamma_{l}\in
		\CH^\ast(\P)$. Substituting this into (\ref{eqn:z}), we get
		\begin{equation}\label{eqn:z2}
		z=\sum_{k}q^{*}(z_{k})\cdot\left(\pi^{*}_{2}(h)-q^{*}(\alpha)\right)^{k}+\sum_{l}\Delta_{*}(p^{*}(\gamma_{l})\cdot\pi^{*}(h)^{l}).
		\end{equation}
		Now for any $b\in B^{\circ}$, the restriction $z|_{S_{b}\times S_{b}}$ is of the
		desired form simply because the restrictions of $\pi_{2}^{*}(h)$ and $p^{*}(h)$
		to the fibers vanish.
	\end{proof}
	
	We can now prove the first two parts of Theorem \ref{thm:mainHilb2}.
	\begin{proof}[Proof of Theorem \ref{thm:mainHilb2} for relative squares]
		Keep the same notations as before. Thanks to Proposition~\ref{prop:Restriction},
		we only need to show that for any  smooth complete intersection K3 surface
		$S\subset \P$, the cycle class map restricted to 
		\[ \im\left(\CH^\ast(\P\times \P)\to \CH^\ast(S\times S)\right)+
		\Delta_{*}\im\left(\CH^\ast(\P)\to \CH^\ast(S)\right)\] 
		is injective. 
		Denote $H\deff c_{1}(\rO_{\P}(1))$ and $h\deff H|_{S}$. Since $\CH^\ast(\P\times
		\P)$ is generated by $\pr_{1}^{*}(H)$ and $\pr_{2}^{*}(H)$, and
		$\Delta_{*}(h)=h\times \mathfrak{o}_{S}+\mathfrak{o}_{S}\times h$ (see
		\cite{BVK3}), it is enough to show that the cycle class map of $S\times S$
		restricted to the subalgebra generated by $\pr_{1}^{*}(h), \pr_{2}^{*}(h)$ and
		$\Delta$ is injective. This is the easiest case of Voisin's \cite[Proposition
		2.2]{MR2435839}.
	\end{proof}
	
	\begin{proof}[Proof of Theorem \ref{thm:mainHilb2} for relative Hilbert squares]
		Consider the blow-up of $\rS^{\circ}\times_{B^{\circ}}\rS^{\circ}$ along the
		relative diagonal $\Delta_{\rS^{\circ}/{B^{\circ}}}$\,; the natural involution
		switching the two factors lifts to the blow-up. It is well-known that the
		Hilbert square is the quotient of this lifted involution and that
		$$\CH^{*}(\Hilb^{2}_{B^{\circ}}(\rS^{\circ}))\isom
		\CH^{*}(\Bl_{\Delta}\left(\rS^{\circ}\times_{B^{\circ}}\rS^{\circ}\right))^{inv}\isom
		\CH^{*}(\rS^{\circ}\times_{B^{\circ}}\rS^{\circ})^{inv}\oplus
		\CH^{*-1}(\rS^{\circ}),$$
		where all isomorphisms are compatible with the restriction to the fibers.
		Therefore, for any $b\in B^{\circ}$, the restriction $z|_{S_{b}^{[2]}}$ of any
		$z\in \CH^{*}(\Hilb^{2}_{B^{\circ}}(\rS^{\circ}))$ to the fiber over $b$, viewed
		as an element in $\CH^{*}(S_{b}\times S_{b})^{inv}\oplus \CH^{*-1}(S_{b})$,
		lives in 
		$\im(\CH^\ast(\rS^{\circ}\times_{B^{\circ}}\rS^{\circ})^{inv}\to
		\CH^{*}(S_{b}\times S_{b})^{inv})\oplus \im(\CH^{*-1}(\rS^{\circ})\to
		\CH^{*-1}(S_{b}))$. We can thus conclude thanks to the established cases of the
		Franchetta property for the relative squares
		$\rS^{\circ}\times_{B^{\circ}}\rS^{\circ}$ and for $\rS^{\circ}$.
	\end{proof}

	\section{Some more cases of Hilbert schemes of K3
		surfaces}\label{sect:HilbK3more}
	In this section, we push the results and methods of \S  \ref{sect:Hilb2K3} to
	higher (Hilbert) powers and to K3 surfaces of higher genera. Let us first
	provide some technical tool for that purpose.
	
	\subsection{Stratified projective bundles}\label{subsect:SPB}
	As one can observe, Lemma \ref{lemma:Restriction} and Proposition
	\ref{prop:Restriction} (but also Proposition \ref{prop:RestIm} below) share some
	similarity. 
	The goal of this technical subsection is to summarize these situations.
	\begin{defi}[Stratified projective bundle]\label{def:SPB}
		A projective morphism $q:\rX\to Y$ is called a \emph{stratified projective
			bundle} if there exists a commutative cartesian diagram
		\begin{equation}\label{diag:Stratification}
		\xymatrix{
			\rX_{r}\ar@{^{(}->}[r]\ar[d]^{q_{r}}\cart& \cdots \ar@{^{(}->}[r] \cart&
			\rX_{1}\ar@{^{(}->}[r] \ar[d]^{q_{1}} \cart& \rX_{0}=\rX \ar[d]^{q_{0}=q}\\
			Y_{r} \ar@{^{(}->}[r]& \cdots \ar@{^{(}->}[r]& Y_{1}\ar@{^{(}->}[r]& Y_{0}=Y
		}
		\end{equation}
		where all horizontal morphisms are closed immersions, such that for any $0\leq
		i\leq r$, the restriction of $q_{i}$ $$q'_{i}: \rX_{i}\backslash \rX_{i+1}\to
		Y_{i}\backslash Y_{i+1}$$ is a projective bundle ($\rX_{r+1}=Y_{r+1}=\vide$).
		The above diagram is called a \emph{stratification} of $q$.
	\end{defi}
	
	Now we can state the following generalization of Lemma \ref{lemma:Restriction}
	and Proposition \ref{prop:Restriction} (see also Proposition \ref{prop:RestIm}
	for an example).
	\begin{prop}\label{prop:SPB}
		Let $q:\rX\to Y$ be a stratified projective bundle with a given stratification
		(\ref{diag:Stratification}) and $\pi: \rX\to B$ be a surjective morphism. Assume
		moreover that for any $0\leq i\leq r$, $Y_{i}$ is smooth projective, $\rX_{i}$
		is flat over a  (common) smooth Zariski open subset $B^{\circ}\subset B$,
		$\codim_{\rX_{i}}(\rX_{i+1})\geq 2$ and finally that
		there exists a line bundle on $B$ whose restriction to fibers of the projective
		bundle $q_{i}'$ is non-trivial. Then for any $b\in B^{\circ}$
		$$\im\left(\CH^\ast(\rX)\to
		\CH^\ast(X_{b})\right)=\sum_{i=0}^{r}\iota_{i*}\im\left(q_{i,b}^{*}:
		\CH^\ast(Y_{i})\to \CH^\ast(X_{i,b})\right),$$
		where $X_{b}$ (\resp $X_{i,b}$) is the fiber of $\rX$ (\resp the Zariski closure
		of $\rX_{i}\backslash \rX_{i+1}$) over $b$, $\iota_{i}: X_{i,b}\inj X_{b}$ is
		the natural inclusion and $q_{i,b}$ is the restriction of $q_{i}$ to $X_{i,b}$.
	\end{prop}
	\begin{proof}
		Since the $\rX_{i}$'s are flat over $B^{\circ}$, by base-change, the right-hand
		side is clearly contained in the left-hand side. We use induction on $r$ to
		prove the other inclusion.
		For any $z\in \CH^\ast(\rX)$, the projective bundle formula shows that
		$$j^{*}(z)=\sum_{k\geq 0} q_{0}'^{*}(z_{k})\cdot \xi^{k},$$ for some cycles
		$z_{k}\in \CH^\ast(Y_{0}\backslash Y_{1})$ where $j: \rX\backslash \rX_{1}\inj
		\rX$ is the open immersion and $\xi=c_{1}(\rO_{q'_{0}}(1))$. By hypothesis,
		$\xi=j^{*}\pi^{*}(h)+{q'_{0}}^{*}(\alpha)$, where $h$ is a divisor on $B$ and
		$\alpha\in \CH^\ast(Y_{0}\backslash Y_{1})$.
		We extend $z_{k}$ and $\alpha$ to $Y_{0}$, keeping the same notation for the
		classes on $Y_0$. Therefore 
		$$z-\sum_{k}q^{*}(z_{k})\cdot\left(\pi^{*}(h)-q^{*}(\alpha)\right)^{k}\in
		\Ker(j^{*}).$$
		By the localization sequence, there exists $\gamma\in \CH^\ast(\rX_{1})$, such
		that 
		\begin{equation}\label{eqn:znew}
		z=\sum_{k}q^{*}(z_{k})\cdot\left(\pi^{*}(h)-q^{*}(\alpha)\right)^{k}+\iota_{*}(\gamma),
		\end{equation}
		where $\iota: \rX_{1}\inj \rX$ is the natural inclusion.
		
		Noting that the restriction of $\pi^{*}(h)$ to $X_{b}$ vanishes, we have that 
		$$z|_{X_{b}}\in \im\left( q^{*}: \CH^\ast(Y)\to
		\CH^\ast(X_{b})\right)+\im\left(\iota_{*}: \CH^\ast(\rX_{1})\to CH^\ast
		(\rX)\right)|_{X_{b}},$$
		where the second term is $\iota_{1,*}\im\left(\CH^\ast(\rX_{1})\to
		\CH^\ast(X_{1,b})\right)$ by flat base-change. Observing that $q_{1}:\rX_{1}\to
		Y_{1}$ is again a stratified projective bundle verifying all the conditions, the
		induction hypothesis allows us to conclude.
	\end{proof}
	
	\subsection{Cubes and Hilbert cubes of complete intersection K3
		surfaces}\label{subsect:Hilb3K3}
	We prove Theorem \ref{thm:mainHilb2} for cubes and Hilbert cubes in this
	subsection. Notation is as in \S \ref{sect:Hilb2K3}. 
	
	The geometry is quite close\footnote{In fact, complete intersection K3 surfaces
		are special cases of Calabi--Yau complete intersections considered in
		\cite{MR3077892} and so all results in \loccit apply.} to the one considered in
	\cite{MR3077892}, in particular, we will study collinear triples in the
	projective space $\P$. For three points in $\P$ there are four types of relative
	positions\,:  non-collinear, collinear and distinct, two coincide but not with
	the third, all coincide. 
	As a result, the evaluation map of the relative cube of the universal family
	\begin{equation*}
	q: \rS\times_{B}\rS\times_{B}\rS\to \P\times \P\times \P
	\end{equation*}
	is not a projective bundle but is a \emph{stratified} projective bundle
	(Definition \ref{def:SPB}) with the following stratification\,:
	\begin{equation}\label{diag:Strat}
	\xymatrix{
		\rS=\delta_{\rS/B} \ar@{^{(}->}[r] \ar[d]^{p} \cart& \Delta_{12}\cup
		\Delta_{13}\cup \Delta_{23} \cart\ar[d]\ar@{^{(}->}[r]& \Delta_{12}\cup
		\Delta_{13}\cup \Delta_{23} \cup \rI \cart \ar[d]\ar@{^{(}->}[r]
		&\rS\times_{B}\rS\times_{B}\rS \ar[r]^-{\pi_{3}} \ar[d]^{q}& B\\
		\P=\delta_{P} \ar@{^{(}->}[r] & \Delta_{12}\cup \Delta_{13}\cup \Delta_{23}
		\ar@{^{(}->}[r]& J \ar@{^{(}->}[r] & \P\times\P\times\P& 
	}
	\end{equation}
	where in the first row, $\Delta_{i,j} : \rS\times_{B}\rS \inj
	\rS\times_{B}\rS\times_{B}\rS$ are the three big (relative) diagonals and $\rI$
	is the Zariski closure of
	$$\rI^{\circ}:=\left\{(x,y,z)\in\rS\times_{B}\rS\times_{B}\rS ~~\vert~~ x,y,z~
	\text{collinear and distinct}\right\}\,;$$ in the second row, $\Delta_{i,j}:
	\P\times \P\inj \P\times \P\times \P$ are the three big diagonals and
	$$J:=\left\{(x,y,z)\in \P\times \P\times \P~~\vert~~ x,y,z~
	\text{collinear}\right\}.$$
	
	\begin{prop}\label{prop:RestrIm3}
		We have for any $b\in B^{\circ}$
		\begin{eqnarray*}
			&&\im\left(\CH^\ast(\rS\times_{B}\rS\times_{B}\rS)\to \CH^\ast(S_{b}\times
			S_{b}\times S_{b})\right)\\&=&\im\left(\CH^\ast(\P\times \P\times \P)\to
			\CH^\ast(S_{b}\times S_{b}\times S_{b})\right)\\&&+ \sum_{1\leq
				i<j\leq3}{\Delta_{i,j}}_{*}\im\left(\CH^\ast(\P\times \P)\to
			\CH^\ast(S_{b}\times S_{b})\right)\\ &&+ \delta_{*}\im\left(\CH^\ast(\P)\to
			\CH^\ast(S_{b})\right),
		\end{eqnarray*}
		where $\Delta_{i,j}: S_{b}^{2}\inj S_{b}^{3}$ are the inclusions of the big
		diagonals and $\delta: S_{b}\inj S_{b}^{3}$ is the inclusion of the small
		diagonal.
	\end{prop}
	\begin{proof}
		It is straight-forward to check that (\ref{diag:Strat}) indeed stratifies $q$
		into projective bundles and that  the codimension of $\rI$ in
		$\rS\times_{B}\rS\times_{B}\rS$ is $\dim(\P)-1$ (\cf \cite[Lemma
		1.2]{MR3077892}), which is $\geq 2$. Moreover, it is clear that
		$\pi_{3}^{*}\rO_{B}(1)$ restricts to the relative ample tautological line bundle
		on fibers of all projective bundles. All assumptions of Proposition
		\ref{prop:SPB} being satisfied, it implies that for any $b\in B^{\circ}$
		\begin{eqnarray*}
			&&\im\left(\CH^\ast(\rS\times_{B}\rS\times_{B}\rS)\to \CH^\ast(S_{b}\times
			S_{b}\times S_{b})\right)\\&=&\im\left(\CH^\ast(\P\times \P\times \P)\to
			\CH^\ast(S_{b}\times S_{b}\times S_{b})\right)+\iota_{*}\im\left(\CH^\ast(J)\to
			\CH^\ast(I_{b})\right)\\
			&&+ \sum_{1\leq i<j\leq3}{\Delta_{i,j}}_{*}\im\left(\CH^\ast(\P\times \P)\to
			\CH^\ast(S_{b}\times S_{b})\right) + \delta_{*}\im\left(\CH^\ast(\P)\to
			\CH^\ast(S_{b})\right),
		\end{eqnarray*}
		where $\iota: I_{b}\inj S_{b}^{3}$ is the inclusion of the Zariski closure of
		the locus of collinear and distinct triples. We only have to show that the
		second term on the right-hand side is redundant. Indeed, for any $b\in
		B^{\circ}$, consider the cartesian square
		\begin{equation*}
		\xymatrix{
			\Delta_{12}\cup \Delta_{13}\cup \Delta_{23}\cup I_{b} \ar@{^{(}->}[r]
			\ar@{^{(}->}[d]  \cart& S_{b}\times S_{b}\times S_{b}\ar@{^{(}->}[d]\\
			J \ar@{^{(}->}[r]& \P\times\P\times \P
		}
		\end{equation*}
		Here
		the intersection is transversal along $I_{b}\backslash\cup\Delta_{i,j}$ (without
		excess intersection) and $\codim_{S_{b}^{3}} I_{b}=\codim_{\P^{\times
				3}}(J)=\dim P-1$, while along $\Delta_{i,j}$ the intersection has excess
		dimension $\dim \P-3$ (\cf \cite[Lemma 1.2]{MR3077892}) with excess normal
		bundle $\frac{\pr_{1}^{*}(E|_{S_{b}})}{\rO(1)\boxtimes \rO(-1)}$ (\cf
		\cite[Lemma 1.5]{MR3077892})\footnote{So there is no excess intersection in the
			case of quartic surfaces.}. The excess intersection class on
		$\Delta_{i,j}=S_{b}\times S_{b}$ is therefore a polynomial in $h_{1}$ and
		$h_{2}$ with $h_{i}:=\pr_{i}^{*}(c_{1}(\rO(1)|_{S_{b}}))$, hence is the
		pull-back of an element in $\CH^\ast(\P\times \P)$. As a result, by the excess
		intersection formula (\cf \cite[\S 6.3]{MR1644323}) applied to the above
		cartesian square, any element in the second term 
		$\iota_{*}\im\left(\CH^\ast(J)\to \CH^\ast(I_{b})\right)$, up to an element in
		the third term $\sum_{1\leq
			i<j\leq3}{\Delta_{i,j}}_{*}\im\left(\CH^\ast(\P\times \P)\to
		\CH^\ast(S_{b}\times S_{b})\right)$, is an element in the first term
		$\im\left(\CH^\ast(\P\times \P\times \P)\to \CH^\ast(S_{b}\times S_{b}\times
		S_{b})\right)$, thus is redundant.
	\end{proof}
	
	We are now ready to prove the remaining cases of Theorem \ref{thm:mainHilb2}\,:
	\begin{proof}[Proof of Theorem \ref{thm:mainHilb2} for relative cubes]
		Denote by $h=c_{1}(\rO_{\P}(1)|_{S_{b}})$ and $h_{i}:=\pr_{i}^{*}(h)$. Thanks to
		Proposition \ref{prop:RestrIm3}, for any $z\in
		\CH^\ast(\rS\times_{B}\rS\times_{B}\rS)$ and any $b\in B^{\circ}$, the
		restriction $z|_{S_{b}\times S_{b}\times S_{b}}$ is a polynomial in $h_{1},
		h_{2}, h_{3}, \Delta_{12}, \Delta_{13}, \Delta_{23}$ (and
		$\delta=\Delta_{12}\Delta_{23}$). We can conclude by the $m=3$ case of Voisin's
		\cite[Proposition 2.2]{MR2435839}, where the essential point is the
		decomposition of the small diagonal $\delta$ due to Beauville--Voisin
		\cite[Proposition 3.2]{BVK3}.
	\end{proof}
	
	\begin{proof}[Proof of Theorem \ref{thm:mainHilb2} for relative Hilbert cubes]
		To simplify the notation, we denote by $S^{[m]}:=\Hilb^{m}(S)$ and similarly
		$\rS^{[m]/B}:=\Hilb^{m}_{B}\rS$.
		Let us first recall the result of de Cataldo--Migliorini \cite{MR1919155} in the
		special case of Chow groups of Hilbert cubes of surfaces\,: for any surface $S$,
		denote by $\rho: S^{[3]}\to S^{(3)}$ the Hilbert--Chow morphism which sends a
		0-dimensional subscheme to its support 0-cycle. We have the incidence
		subvarieties
		\begin{eqnarray*}
			U&:=&\left\{(z, x_{1}, x_{2}, x_{3})\in S^{[3]}\times S^{3}~~\vert~~
			\rho(z)=x_{1}+x_{2}+x_{3}\right\}\,;\\
			V&:=&\left\{(z, x_{1}, x_{2})\in S^{[3]}\times S^{2}~~\vert~~
			\rho(z)=2x_{1}+x_{2}\right\}\,;\\
			W&:=&\left\{(z, x)\in S^{[3]}\times S~~\vert~~ \rho(z)=3x\right\}\,;
		\end{eqnarray*}
		and the main result of \cite{MR1919155} says that together they induce an
		injective morphism
		$$(U_{*}, V_{*}, W_{*}): \CH^\ast(S^{[3]})\inj \CH^\ast(S^{3})\oplus
		\CH^\ast(S^{2})\oplus \CH^\ast(S).$$
		Note that the above correspondences have natural family counterparts, denoted by
		$\rU, \rV, \rW$.
		
		Let $z\in \CH^\ast(\rS^{[3]/B})$ be such that the cohomology class of
		$z|_{S_{b}^{[3]}}$ vanishes. By the above injectivity, it is enough to show that
		for any $b\in B^{\circ}$,  $U_{*}\left(z|_{S_{b}^{[3]}}\right)$,
		$V_{*}\left(z|_{S_{b}^{[3]}}\right)$ and $W_{*}\left(z|_{S_{b}^{[3]}}\right)$
		are zero. To this end, observe that
		$U_{*}\left(z|_{S_{b}^{[3]}}\right)=\rU_{*}(z)|_{S_{b}^{3}}$ is the restriction
		of a cycle of the total family $\rS\times_{B}\rS\times_{B}\rS$ with trivial
		cohomology class, hence is zero by the relative cube case of Theorem
		\ref{thm:mainHilb2} just proven. Similarly, the vanishing of
		$V_{*}\left(z|_{S_{b}^{[3]}}\right)$ and $W_{*}\left(z|_{S_{b}^{[3]}}\right)$
		follow from the relative square case proven in \S\ref{sect:Hilb2K3} and
		\cite{FranchettaK3} respectively.\\
		Finally, the proof of the case of $\rS\times_{B} \rS^{[2]/B}$ is similar (in
		fact, easier) by using the motivic decomposition for Hilbert squares.
	\end{proof}
	
	\subsection{Beyond complete intersection K3 surfaces}\label{subsect:beyond}
	The techniques we utilized above in order to prove Theorem \ref{thm:mainHilb2}
	for (Hilbert) squares and cubes of complete intersection K3 surfaces can also be
	employed to attack the generalized Franchetta conjecture \ref{conj:FranchettaHK}
	for families of K3 surfaces for which Mukai models are available. In this
	subsection, we give a sufficient condition for the Franchetta property to hold
	for Hilbert schemes of K3 surfaces in a certain range. It is convenient to
	introduce the following notion\,:
	
	\begin{defi}[Tautological ring]\label{def:TautoRing}
		Let $(S, H)$ be a polarized K3 surface and $r\in\N$. Denote $h:=c_{1}(H)\in
		\CH^{1}(S)$. The  \emph{tautological ring} $R^\ast(S^{r})$ is the subring of the
		(rational) Chow ring $\CH^\ast(S^{r})$ generated by the big diagonals
		$\Delta_{i,j}$ ($1\leq i<j \leq r$), the polarization classes
		$h_{i}:=\pr_{i}^{*}(h)$ and the Beauville--Voisin classes $\mathfrak
		o_{i}:=\pr_{i}^{*}(\mathfrak o_{S})$ ($1\leq i\leq r$).
	\end{defi}
	
	\begin{rmk}\label{rmk:Rpreserved}
		Using \cite[Proposition 2.6]{BVK3}, we see that the tautological rings of
		different powers of a K3 surface are closed under push-forwards and pull-backs
		along all kinds of (partial) diagonal inclusions.
	\end{rmk}
	
	Recall that for a natural number $g$, we say that a \emph{Mukai model} for K3
	surfaces of genus $g$ exists, if there exist an ambient homogeneous space 
	$G=G_{g}$ (often a Grassmannian) and a globally generated homogeneous vector
	bundle $E=E_{g}$ on $G$ such that the zero locus of a general section of $E$
	gives a general K3 surface of genus $g$. For the available constructions of
	Mukai models and the corresponding $G$ and $E$, we refer to \cite{FranchettaK3}
	as well as the original sources \cite{Mukai88}, \cite{Mukai89}, \cite{Mukai06},
	\cite{Mukai12}. Accordingly, we have a universal family
	\begin{equation*}
	\xymatrix{
		\rS \ar[r]^{p} \ar[d]_{\pi} & G\\
		B=H^0(G,E)&
	} 
	\end{equation*}
	and we denote $B^\circ \subset B$ the locus parameterizing smooth K3 surfaces of
	genus $g$.
	
	The crucial condition for our techniques to work is the following\,:
	\begin{defi}
		For an $r\in \N^{*}$, we say that the Mukai model $(G, E)$ satisfies the
		condition $(\star_{r})$ if
		
		$(\star_{r}) :$ for any $x_{1}, \cdots, x_{r}$ distinct points of $G$, the
		following evaluation map is surjective $$H^{0}(G, E)\to \bigoplus_{i=1}^{r}
		E_{x_{i}}.$$
		Or equivalently, $H^{0}(G, E\otimes I_{x_{1}}\otimes\cdots\otimes I_{x_{r}})$ is
		of codimension $r\cdot \rank(E)$ in $H^{0}(G, E)$.\\ Clearly, $(\star_{r})$
		implies $(\star_{k})$ for all $k<r$.
	\end{defi}

	\begin{prop}\label{prop:beyond}
		The notation is as above. Fix a genus $g$ for which a Mukai model exists for K3
		surfaces of genus $g$ and fix such a Mukai model which satisfies
		condition $(\star_{r})$. Assume that
		$$\im\left(\CH^\ast(\rS)\to \CH^\ast(S_{b})\right)=R^\ast(S_{b}),$$ for any
		$b\in B^{\circ}$.
		Then
		$$\im\left(\CH^\ast(\rS^{r/B})\to
		\CH^\ast(S_{b}^{r})\right)=R^\ast(S_{b}^{r}),$$ 
		for any $b\in B^{\circ}$.
	\end{prop}
	\begin{proof}
		The proof is to rephrase every step of \S \ref{subsect:Hilb3K3} in the general
		setting. We proceed by induction on~$r$.
		Consider the evaluation map $q: \rS^{r/B}\to G^{r}$, which is a stratified
		projective bundle (Definition~\ref{def:SPB}) with the stratification on $G^{r}$
		given by the different types of incidence relations for $r$ points of $G$\,:
		\begin{equation}
		\xymatrix{
			\rX_{n}=\rS\ar@{^{(}->}[r]\ar[d]^{q_{n}=p}\cart& \cdots \ar@{^{(}->}[r] \cart&
			\rX_{1}\ar@{^{(}->}[r] \ar[d]^{q_{1}} \cart& \rX_{0}=\rS^{r/B} \ar[d]^{q_{0}=q}
			\ar[r]& B\\
			Y_{n}=G \ar@{^{(}->}[r]& \cdots \ar@{^{(}->}[r]& Y_{1}\ar@{^{(}->}[r]&
			Y_{0}=G^{r}&
		}
		\end{equation}
		By Proposition \ref{prop:SPB}, for any $b\in B^{\circ}$,
		\begin{equation}\label{eqn:Image}
		\im\left(\CH^\ast(\rS^{r/B})\to
		\CH^\ast(S_{b}^{r})\right)=\sum_{i=0}^{n}\iota_{i,*}\im\left(\CH^\ast(Y_{i})\to
		\CH^\ast({\rX_{i}}'_{b})\right),
		\end{equation}
		where $\rX_{i}'$ is the Zariski closure of $\rX_{i}\backslash \rX_{i+1}$.
		Let us show that each term of (\ref{eqn:Image}) is in the tautological ring
		$R^\ast(S_{b})$ 
		by ascending order for $0\leq i\leq n$\,:
		
		\begin{itemize}
			\item If $i=0$, since the Chow ring of $G$ satisfies the K\"unneth formula, we
			only need to show that $$\im\left(\CH^\ast(G)\to \CH^\ast(S_{b})\right)\subset
			R^\ast(S_{b}),$$
			which is true by assumption.
			\item If a general point of $Y_{i}$ is parameterizing $r$ points of $G$ where at
			least two of them coincide, then the contribution of the $i$-th term of
			(\ref{eqn:Image}) factors through $R^{*}(S_{b}^{r-1})$ (via the diagonal
			push-forward) by the induction hypothesis, hence is contained in
			$R^{*}(S_{b}^{r})$ (Remark \ref{rmk:Rpreserved}).
			\item If a general point of $Y_{i}$ is parameterizing $r$ different points of
			$G$, then the hypothesis $(\star_{r})$ means precisely that any $r$ different
			points of $G$ impose independent conditions on $B$, each of codimension
			$\rank(E)$. Therefore, $\rX_{i}'$, the Zariski closure of $\rX_{i}\backslash
			\rX_{i+1}$, has the same codimension in $\rX_{i-1}$ as $\codim_{Y_{i-1}}
			(Y_{i})$. The excess intersection formula (\cite[\S 6.3]{MR1644323}) applied to
			the cartesian diagram
			\begin{equation*}
			\xymatrix{
				\rX_{i}=\rX_{i+1}\cup \rX'_{i}\ar@{^{(}->}[r] \ar[d]  \cart& \rX_{i-1}\ar[d]\\
				Y_{i} \ar@{^{(}->}[r]& Y_{i-1}
			}
			\end{equation*}
			tells us that modulo the $(i+1)$-th term of (\ref{eqn:Image}), the contribution
			of the $i$-th term is contained in the $(i-1)$-th term. 
		\end{itemize}
	\end{proof}
	
	\begin{thm}\label{thm:beyond}
		Fix a genus $g$ for which a Mukai model exists for K3 surfaces of genus $g$, and
		fix such a Mukai model.  Assume that
		\begin{enumerate}[(i)]
			\item
			the Mukai model satisfies the condition $(\star_{r})$\,; 
			\item Conjecture \ref{conj:FranchettaK3} is true for the universal family
			$\rS\to B$ of K3 surfaces of genus $g$\,;
			\item the cycle class map restricted to the tautological ring $R^{*}(S^{r})$ is
			injective for the very general K3 surface $S$ of genus $g$.
		\end{enumerate}
		Then the Franchetta property holds for
		$\rS^{[r_{1}]/B}\times_{B}\cdots\times_{B} \rS^{[r_{m}]/B}$, for any $r_{1},
		\cdots, r_{m}$ whose sum is $\leq r$.
	\end{thm}
	\begin{proof}
		The case of relative powers $\rS^{k/B}$, for any $k\leq r$, is a direct
		consequence of Proposition \ref{prop:beyond} and the hypothesis on the
		injectivity of the cycle class map on the tautological ring. The other cases
		reduce to the cases of $\rS^{k/B}$ for all $1\leq k\leq r$ by making use of de
		Cataldo--Migliorini's result \cite{MR1919155} for Chow motives of Hilbert
		schemes of surfaces.
	\end{proof}

	We apply Theorem \ref{thm:beyond} to some Mukai models to get concrete
	unconditional results\,:
	\begin{proof}[Proof of Theorem \ref{thm:mainHilbmore}]
		Assumption $(ii)$ is proven for $g\in\{2,\ldots,10\}\cup\{12\}$ in
		\cite{FranchettaK3}. Assumption $(iii)$ is taken care of for $r\le 43$ by
		Voisin's \cite[Proposition 2.2]{MR2435839}.
		It remains to check assumption $(i)$ of Theorem \ref{thm:beyond}\,; we proceed
		by a case-by-case analysis of the positivity of the homogeneous bundle in the
		Mukai model. See Mukai's series of papers \cite{Mukai88}, \cite{Mukai89},
		\cite{Mukai06}, \cite{Mukai12} for more information on the geometry of these
		models.
		\begin{itemize}
			\item  K3 surfaces of genus $g=2$ are{}\footnote{Equivalently, these K3 surfaces
				are also double covers of $\P^{2}$ ramified along smooth sextic curves.} smooth
			degree 6 hypersurfaces in the weighted projective space $\P:=\P(1,1,1,3)$. The
			Mukai model for this family is thus $(G, E)=(\P, \rO(6))$. Note that the K3
			surfaces in this family all avoid the singular point $O:=[0,0,0,1]$.  Let us
			check the condition $(\star_{3})$, \ie, that the evaluation map $$H^{0}(\P,
			\rO(6))\to \bigoplus_{i=1}^{3}\C_{x_{i}}$$ is surjective for distinct $x_{1},
			x_{2}, x_{3}\neq O$, where $\C_{x}$ denotes the fiber of $\rO(6)$ at $x$. It is
			easy to see that $\P(1,1,1,3)$ is isomorphic to the projective cone over the
			third Veronese embedding of $\P^{2}$ (\cf \cite{Dolgachev}) and $O$ is the
			vertex. By upper-semicontinuity, it is enough to treat the most degenerate case
			for three distinct points of $\P\backslash \{O\}$, which is when they lie in the
			same ruling of the projective cone. In this case, as the restriction of $\rO(6)$
			to the ruling is $\rO(2)$, the condition $(\star_{3})$ follows from the
			surjections\,:
			$$H^{0}(\P, \rO(6))\surj H^{0}(\P^{1}, \rO_{\P^{1}}(2))\surj
			\bigoplus_{i=1}^{3}\C_{x_{i}},$$ where $\P^{1}$ is the ruling which contains
			$x_{i}$'s. 
			\item For quartic surfaces ($g=3$), let us first show that $(\P^{3}, \rO(4))$
			satisfies $(\star_{5})$, \ie, that the evaluation map $$H^{0}(\P^{3}, \rO(4))\to
			\bigoplus_{i=1}^{5}\C_{x_{i}}$$
			is surjective for distinct $x_{i}$'s. Again, it is enough to treat the most
			degenerate cases, namely\,:
			\begin{itemize}
				\item when $x_{1}, \cdots, x_{5}$ are collinear, then this follows from the
				surjectivity of the restriction and the evaluation $$H^{0}(\P^{3}, \rO(4))\surj
				H^{0}(\P^{1}, \rO(4))\surj \bigoplus_{i=1}^{5}\C_{x_{i}},$$ where $\P^{1}$ is
				the line containing these points.
				\item when $x_{1}, \cdots, x_{5}$ are in a conic $C$. Then the Koszul resolution
				provides an exact sequence $$0\to \rO_{\P^{3}}(-3)\to \rO_{\P^{3}}(-1)\oplus
				\rO_{\P^{3}}(-2)\to\rO_{\P^{3}}\to \rO_{C}\to 0,$$ which allows us to see that
				the restriction map $H^{0}(\P^{3}, \rO(4))\to H^{0}(C, \rO_{C}(8))$ is
				surjective. Since $H^{0}(C, \rO_{C}(8))\to \bigoplus_{i=1}^{5}\C_{x_{i}}$ is
				clearly surjective, we are done.
			\end{itemize}
			The condition $(\star_{5})$ is proven. 
			\item For $g=6$, the Mukai model is $(G, E)=(\Gr(2,5), \rO(1)^{\oplus 3}\oplus
			\rO(2))$, where $\rO(1)$ is the Pl\"ucker line bundle. It is clear that the
			condition $(\star_{2})$ is equivalent to the surjectivity of $$H^{0}(G,
			\rO(1))\to \C_{x_{1}}\oplus\C_{x_{2}}$$ for any two distinct points $x_{1},
			x_{2}\in G$. This last condition follows from the very ampleness of the
			Pl\"ucker line bundle $\rO(1)$.
			\item For $g=7$, the Mukai model is $(G, E)=\left(\OGr(5,10), U^{\oplus
				8}\right)$, where $\OGr(5,10)$ is the orthogonal Grassmannian parameterizing
			isotropic subspaces of dimension 5 in a vector space of dimension 10 equipped
			with a non-degenerate quadratic form and $U$ is a line bundle corresponding to
			the half spinor representation. The proof is similar to the previous case\,: one
			uses the very ampleness of $U$.
			\item For $g=8$, the Mukai model is $(G, E)=\left(\Gr(2,6), \rO(1)^{\oplus
				6}\right)$, where $\rO(1)$ is the Pl\"ucker line bundle. The proof goes as for
			$g=6$ by the very ampleness of the Pl\"ucker line bundle.
			\item For $g=9$, the Mukai model is $(G,E)=\left(\LGr(3,6),\rO(1)^{\oplus
				4}\right)$, where $\LGr(3,6)$ is the symplectic Grassmannian parameterizing
			Lagrangian subspaces in a 6-dimensional vector space equipped with a symplectic
			form and $\rO(1)$ is the restriction of the Pl\"ucker line bundle of $\Gr(3,6)$.
			The proof goes as before\,: one uses the very ampleness of $\rO(1)$.
			\item For $g=10$, the Mukai model is $(G, E)=(G_{2}/P, \rO(1)^{\oplus 3})$,
			where $G$ is the 5-dimensional quotient of the simply-connected semi-simple
			algebraic group of type $G_{2}$ by a maximal parabolic subgroup $P$ and $\rO(1)$
			is the line bundle associated to the adjoint representation of $G_{2}$\,; in
			other words, $G=G_{2}/P\inj \P(\mathfrak{g}_{2}^{\vee})$. Again, we can conclude
			by the very ampleness of $\rO(1)$.
			\item For $g=12$, we use a slight variant of the above argument. Indeed, the
			general K3 surface of genus 12 can be constructed as an anti-canonical section
			in a smooth prime Fano threefold $X$ of genus 12 (\cf \cite{Beau}, \cite[Section
			3.1]{IM}). The Fano threefold $X$ has very ample anti-canonical bundle, and
			$H^3(X,\Q)=0$ (\cite[Corollary 4.3.5]{IP}) so that $X$ has trivial Chow
			groups\footnote{Following Voisin \cite{MR3099982}, we say a smooth projective
				variety has \emph{trivial Chow groups} if the cycle map $\cl^{i}:
				\CH^{i}(X)_{\Q}\to H^{2i}(X, \Q)$ is injective for any $i$.} (this Fano
			threefold $X$ is the variety denoted by $X_{22}\subset \P^{13}$ in
			\cite[Propositions 4.1.11 and 4.1.12]{IP}\,; actually $X$ is an intersection of
			quadrics). We now consider a variant of Theorem \ref{thm:beyond}, replacing $G$
			by $X$ and $E$ by $-K_X$. The very ampleness of $-K_X$ ensures that condition
			$(\star_2)$ holds. As $X$ has trivial Chow groups, there is a Chow--K\"unneth
			formula for products of $X$, and so one is reduced to the statement for the K3
			surface $S_b$, which is \cite{FranchettaK3}.
		\end{itemize}
	\end{proof}

	\begin{rmk}[Limit of our method]\label{rmk:limit}
		Given a Mukai model $(G, E)$, 
		\begin{itemize}
			\item the global generation of $E$ corresponds to condition $(\star_{1})$, which
			essentially explains the reason why one can prove the generalized Franchetta
			conjecture for K3 surfaces with a Mukai model in \cite{FranchettaK3}. 
			\item For K3 surfaces of genus 2, $G=\P(1,1,1,3)$ and $E=\rO(6)$, the condition
			$(\star_{4})$ is not satisfied\,: it is violated by three distinct points lying
			on the same ruling, away from the singular point.
			\item For the quartic K3 surfaces, $G=\P^{3}$ and $E=\rO(4)$, the condition
			$(\star_{6})$ is not satisfied\,: it is violated by six collinear distinct
			points. Similarly, for the other two families of complete intersection K3
			surfaces (genus 4 and 5), $(\star_{4})$ is violated by four collinear distinct
			points. 
			\item For K3 surfaces of genus 6 and 8, whose Mukai model is $(G, E)=(\Gr(2,5),
			\rO(1)^{\oplus 3}\oplus \rO(2))$ and $(\Gr(2,6), \rO(1)^{\oplus 6})$
			respectively, the condition $(\star_{3})$ is not satisfied. Indeed, it is
			equivalent to the surjectivity of $H^{0}(G, \rO(1))\to
			\C_{x_{1}}\oplus\C_{x_{2}}\oplus \C_{x_{3}}$, which is violated by three
			distinct collinear points of $G$.
			\item For K3 surfaces of genus 13 and 20, the Mukai models are respectively
			\begin{eqnarray*}
				(G,E)=  \left(\Gr(3,7), (\wedge^{2}S^{\vee})^{\oplus 2}\oplus \wedge^{3}Q\right)
				\text{  and  } \left(\Gr(4,9), (\wedge^{2}S^{\vee})^{\oplus 3}\right).
			\end{eqnarray*}
			where $S$ is the tautological subbundle and $Q$ is the tautological quotient
			bundle. We claim that none of them verifies the condition $(\star_{2})$. For
			example, in the genus 13 case, the condition $(\star_{2})$ is equivalent to the
			surjectivities of the following two evaluation maps
			$$H^{0}(G, \wedge^{2}S^{\vee})\to \wedge^{2}S_{x}^{\vee}\oplus
			\wedge^{2}S_{y}^{\vee},$$
			$$H^{0}(G, \wedge^{2}Q)\to \wedge^{2}Q_{x}\oplus \wedge^{2}Q_{y},$$
			for any $x\neq y\in G$, which, by Bott theorem, amount to say that for any two
			different 3-dimensional subspaces  $W_{1}, W_{2}$ in a 7-dimensional vector
			space $V$, the natural maps $$\wedge^{2}V^{\vee}\to
			\wedge^{2}(W_{1}^{\vee})\oplus \wedge^{2}(W_{2}^{\vee}) $$ $$\wedge^{2}V\to
			\wedge^{2}(V/W_{1})\oplus \wedge^{2}(V/W_{2})$$ are surjective. It is not true
			when $\dim W_{1}\cap W_{2}\geq 2$. The case of genus 20 is similar.
			\item For K3 surfaces of genus 18, the Mukai model is $(G, E)=(\OGr(3,9),
			U^{\oplus 5})$, where $U$ is the rank 2 vector bundle associated to the
			representation $V$  associated to the fourth dominant weight
			$\omega_{4}=\frac{1}{2}(\alpha_{1}+2\alpha_{2}+3\alpha_{3}+4\alpha_{4})$, of the
			(semi-simple part of the) maximal parabolic group $P$. We claim that
			$(\star_{2})$ does not hold, \ie, there exist two different points $x, y\in G$
			such that $H^{0}(G, U)\to U_{x}\oplus U_{y}$ is not surjective\footnote{We thank
				Nicolas Ressayre for his kind help on the proof.}. Let $x=P/P$ and $y=wP/P$
			where $w=s_{\alpha_{3}}$, as an element in the Weyl group $W$, is the reflection
			with respect to the third simple root. Clearly, $w$ does not belong to the Weyl
			group of $P$, which is generated by $s_{\alpha_{1}}, s_{\alpha_{2}},
			s_{\alpha_{4}}$. A direct computation shows that the representation $H^{0}(G,
			U)$ has multiplicity one for all weights. On the other hand, $\omega_{4}$ is a
			common weight for $V$ and its conjugate by $w$ (since $w.
			\omega_{4}=\omega_{4}$). Hence $H^{0}(G, U)\to U_{x}\oplus U_{y}$ cannot be
			surjective.
			\item If one wants to follow the same strategy of this paper to establish the
			Franchetta property for (Hilbert) powers beyond the range stated in Theorem
			\ref{thm:mainHilb2} and Theorem~\ref{thm:mainHilbmore}, one has to deal with
			some essentially new universal cycles, which may not belong to the tautological
			ring, or rather, the tautological ring should be enlarged to include some more
			incidence classes from projective geometry than just the polarization class. 
		\end{itemize}
	\end{rmk}
	
	\subsection{Applications towards the Beauville--Voisin conjecture}
	
	Let us now turn to the consequences of our results in the direction of the
	Beauville--Voisin conjecture (and its refined version
	Conjecture~\ref{conj:BV+})\,:
	\begin{proof}[Proof of Corollaries \ref{cor:BV+} and \ref{cor:BV++}]
		The strongly coisotropic subvarieties $E_\mu$, and the
		Lagrangian surfaces $T$ and $U$, can all be defined over (suitable relative
		powers of) the universal family, and so these are just special cases of
		Proposition \ref{prop:ChernBV}, combined with Theorems \ref{thm:mainHilb2} and
		\ref{thm:mainHilbmore}.
	\end{proof}

	One can also prove a version of Corollaries \ref{cor:BV+} and \ref{cor:BV++} for
	product varieties of arbitrarily high dimension, but the statement is now
	restricted to $0$-cycles and $1$-cycles\,:
	
	\begin{cor}\label{cor:BV+r} Let $X$ be a product 
		\[ X=X_1\times X_2\times\cdots\times X_s,\ \ \ \dim X=2m,\]
		where $X_j$ is a Hilbert scheme $S^{[r]}$ with $S$ a K3 surface. Let $\tilde
		R^\ast(X)\subset \CH^\ast(X)$ denote\footnote{In this paper, the notation
			$R^\ast(X)$ is reserved for the \emph{tautological ring} of a power of K3
			surface, see Definition \ref{def:TautoRing}.} the 
		$\Q$-subalgebra generated by (pullbacks of) divisors on $X_j$, the Chern classes
		$c_{i}(T_{X_j})$, 
		plus the following coisotropic subvarieties\,:
		
		\begin{itemize}
			\item the strongly coisotropic subvarieties $E_\mu$ of \cite[4.1 item
			1)]{MR3524175}\,;
			\item
			the Lagrangian surfaces $T\subset X_j$ constructed in \cite[Proposition 4]{IM}
			(if $X_j=S^{[2]}$ and $S$ is of genus $g\in\{2,3,4,5,6,7,8,9,10,12\}$)\,;
			\item
			the surface of bitangents $U\subset X_j$ (if $X_j=S^{[2]}$ and $S$ is a quartic
			K3 surface).
		\end{itemize}
		Then $\tilde R^{2m}(X)$ and $\tilde R^{2m-1}(X)$ inject into cohomology via the
		cycle class map.
	\end{cor}

	\begin{proof}[Proof of Corollary \ref{cor:BV+r}] This uses the fact that the
		$X_j$ have a {\em multiplicative Chow--K\"unneth decomposition\/}
		$\{\pi^k_{X_j}\}$, in the sense of \cite[Chapter 8]{SV}, \cite{SV2}\,; see also
		Appendix \ref{sec:appfourier}. This induces a bigrading of the Chow ring of
		$X_j$, given by 
		\[ \CH^i(X_j)_{(k)}:=(\pi_{X_j}^{2i-k})_\ast \CH^i(X_j).\] 
		It is readily seen that the projectors $\pi^k_{X_j}$ are universally defined
		(\ie, they exist as a relative cycle for the family $\rX^\circ_j\times_{B^\circ}
		\rX^\circ_j$). Theorem \ref{thm:beyond} applied to the relative cycle
		$\rT-(\pi^2_{\rX_j})_\ast(\rT)$ (where we use the formalism of relative
		correspondences as in \cite[Section 8.1]{MNP}), thus implies that
		\[ T\ \in\ \CH^2(X_j)_{(0)}.\]
		Similarly, we find that $U\in \CH^2(X_j)_{(0)}$. 
		The fact that $E_\mu$ belongs to $\CH^\ast(X_j)_{(0)}$ is true for Hilbert
		schemes of arbitrary K3 surfaces, \cf \cite[Lemma 4.3]{MR3524175}.
		
		The product $X$ also has a multiplicative Chow--K\"unneth decomposition, and
		hence there is a bigrading of the Chow ring $\CH^\ast(X)$ by \cite[Theorem
		8.6]{SV}. Since divisors and Chern classes of $X_j$ are also in
		$\CH(X_j)_{(0)}$, and pullback under any projection $X\to X_j$ preserves the
		bigrading \cite[Corollary 1.6]{SV2}, we see that there is an inclusion
		\[ \tilde R^\ast(X)\ \subset\CH^\ast(X)_{(0)}.\]
		The corollary now follows, since it is known that $\CH^i(X)_{(0)}$ injects
		into cohomology for $i\ge \dim (X)-1$, see \cite[Introduction]{motiveHK}.
	\end{proof}

	\subsection{Double EPW sextics}\label{epw} The interested reader will have no
	trouble finding further applications in the flavour of Corollaries \ref{cor:BV+}
	and \ref{cor:BV++}. For instance, consider the Hilbert square $X=S^{[2]}$, where
	$S$ is a general K3 surface of genus $6$. As shown by O'Grady \cite[Section
	4]{OG}, $X$ is isomorphic to a small resolution $X_A^\epsilon$ of a singular
	double EPW sextic $X_A$ (notation is as in \cite{OG}). Let 
	$\epsilon \colon X\to X_A$ denote the small resolution, and let $f_A\colon
	X_A\to Y_A$ denote the double cover to the associated EPW sextic 
	$Y_A$. The surface $S$ being general corresponds to the fact that the
	Lagrangian vector space $A$ is general (in the precise sense given in \cite[\S
	4]{OG}) in the divisor $\Delta\subset  \LGr(\wedge^3 V)$ studied in \cite{OG}. 
	This construction produces Lagrangian surfaces in $X$\,: the surface 
	\[P:=\epsilon^{-1}(\hbox{Sing}(X_A))\] 
	(which is isomorphic to $\P^2$ since $X_A$ has only one singular point), and
	the surface 
	\[\hbox{Fix}:=\epsilon^{-1}(\hbox{Fix}(\iota)),\] 
	where $\hbox{Fix}(\iota)$ denotes the fixed point locus of the
	(anti--symplectic) covering involution $\iota$ of $X_A$. These Lagrangian
	surfaces are easily seen to be universally defined. (Indeed, as shown in
	\cite{OG}, there is a stratification 
	\[Y_A[3]\subset Y_A[2]\subset Y_A[1]=Y_A\] 
	of the EPW sextic $Y_A$. Here the surface $Y_A[2]$ is the singular locus of
	$Y_A$ and the point $Y_A[3]$ is the unique singularity of $Y_A[2]$. One has 
	\[ \hbox{Fix}=(f_A\circ \epsilon)^{-1}(Y_A[2])\  \ \text{and} \ \  P=
	(f_A\circ \epsilon)^{-1}(Y_A[3]).\]
	On the other hand (as explained in \cite[Section 3]{OG}), there exist family
	versions $\rY[i]$ of the subvarieties $Y_A[i]$ over the base $\Delta$. One can
	perform a base change
	\[ \begin{array} [c]{ccc}
	\bar{\rX}_{B^\circ} & \to& \bar{\rX}\\
	\downarrow&&\downarrow\\
	{\rY}_{B^\circ} & \to& {\rY}\\
	\downarrow&&\downarrow\\
	B^\circ& \to& \Delta\\
	\end{array}\]
	where $B^\circ\subset B$ is an open such that the rational map
	$B\dashrightarrow \Delta$ of \cite[Section 4]{OG} is defined, and $ \bar{\rX}$
	is the tautological family of singular double EPW sextics over $\Delta$. 
	
	One obtains relative versions of $P$ and of $\hbox{Fix}$ by pulling
	back  $\rY[i]$ under the birational morphism $\rX_{B^\circ}\to
	\bar{\rX}_{B^\circ} $.) Thus, 
	applying Theorem \ref{thm:mainHilbmore} one obtains the following\,:
	
	\begin{cor}\label{cor:doubleEPW}
		Let $X=S^{[2]}$, where $S$ is a general K3 surface of genus $6$. The
		$\Q$-subalgebra
		\[  < D_1, D_2, c_i(T_X), P, \hbox{Fix}, T>\  \  \  \subset\ \CH^\ast(X) \]
		injects into cohomology via the cycle class map. (Here $D_1, D_2$ are two
		divisors generating the Picard group of $X$, and $T$ is the Lagrangian surface
		of \cite[Proposition 4]{IM}.)
	\end{cor}

	\subsection{An application to Bloch's conjecture }\label{subsect:AppBloch}
	Given a quartic K3 surface $S$, Beauville \cite{Kata} constructed an interesting
	involution $\iota$ on $X:=S^{[2]}$, which, generically, sends $\{x_{1}, x_{2}\}$
	to $\{x_{3}, x_{4}\}$, where $x_{1}, \ldots, x_{4}$ are the four intersection
	points of the line $\bar{x_{1}, x_{2}}$ with $S$. The involution $\iota$ is
	anti-symplectic. According to the generalized Bloch conjecture (\cf \cite[\S
	11.2]{MR1997577}), which roughly says that $\CH_{0}$ is ``controlled'' by the
	holomorphic forms, the action of $\iota$ on $\CH_{0}(X)$ should be the identity
	on $\Gr_{F}^{0}\CH_{0}(X)$ and on $\Gr^{4}_F\CH_{0}(X)$ (just as on $H^{0}(X)$
	and $H^{4,0}(X)$) and should be $-\id$ on $\Gr_{F}^{2}\CH_{0}(X)$ (just as on
	$H^{2,0}(X)$), where $F^{\cdot}$ is the conjectural Bloch--Beilinson filtration.
	On the other hand, as  conjectured in \cite{MR2187148} by Beauville and worked
	out by Shen--Vial in \cite{SV} in the case of Hilbert squares of K3 surfaces, we
	have a canonical splitting of this filtration for $X$, giving a direct sum
	decomposition\,:
	$$\CH^{4}(X)=\CH^{4}(X)_{(0)}\oplus\CH^{4}(X)_{(2)}\oplus\CH^{4}(X)_{(4)}.$$
	Hence the action of $\iota$ on the three summands should be $\id$, $-\id$ and
	$\id$, respectively. Our results allow us to confirm this expectation.
	
	\begin{proof}[Proof of Corollary \ref{cor:blochkata}] 
		Let $\rS^{\circ}\to B^{\circ}$ be the universal family of smooth quartic K3
		surfaces and $\rX^{\circ}\to B^{\circ}$ be the relative Hilbert square.
		As noted above, the bigrading $\CH^\ast(X)_{(\ast)}$ is induced by a self-dual
		multiplicative Chow--K\"unneth decomposition $\{\pi^k_X\}$ that is universally
		defined. The anti-symplectic involution $\iota$ can also be defined on the level
		of the universal family\,; let us denote $\Gamma_\iota\in
		\CH^4(\rX^\circ\times_{B^\circ}\rX^\circ)$ the graph of the involution
		$\iota\colon\rX^\circ\to\rX^\circ$.
		
		The relative correspondence
		\[ \pi^i_{\rX} \circ  \Gamma_\iota\circ \pi^j_{\rX} \ \ \ \in\
		\CH^4(\rX^\circ\times_{B^\circ}\rX^\circ) \]
		is fiberwise homologically trivial for $i\neq j$. Theorem
		\ref{thm:mainHilbmore} $(ii)$ for $\Hilb^{2}_{B}\rS\times_{B}\Hilb^{2}_{B}\rS$
		implies that
		\begin{equation}\label{comm}  \Bigl(\pi^i_{\rX} \circ \Gamma_\iota\circ
		\pi^j_{\rX} \Bigr)\vert_{X_b\times X_b}=0\ \ \ \hbox{in}\ \CH^4(X_b\times X_b),\
		\ \ \forall i\neq j \ \ \forall b\in B^\circ\ , \end{equation}
		\ie, $\Gamma_{\iota_b}$ belongs to $\CH^4(X_b\times X_b)_{(0)}$, and thus
		$\iota_b$ preserves the bigrading $\CH^\ast(X_b)_{(\ast)}$.
		
		Next, the fact that $\iota_b$ is anti-symplectic means that for any $b\in
		B^\circ$ there exists a divisor $D_b\subset X_b$, and a cycle $\gamma_b$
		supported on $D_b\times D_b$, such that
		\[  \Bigl( (\Delta_{\rX}+\Gamma_\iota)\circ \pi^2_{\rX} \Bigr)\vert_{X_b\times
			X_b}=\gamma_b\ \ \ \hbox{in}\ H^8(X_b\times X_b).\]
		Using a Hilbert schemes argument as in \cite[Proposition 3.7]{MR3099982}, the
		$D_b$ and $\gamma_b$ can be spread out, \ie, there exists a divisor $\rD\subset
		\rX$ and a relative cycle $\gamma$ supported on $\rD\times_{B^\circ}\rD$ such
		that
		\[   \Bigl((\Delta_{\rX}+\Gamma_\iota)\circ \pi^2_{\rX}
		-\gamma\Bigr)\vert_{X_b\times X_b}=0\ \ \ \hbox{in}\ H^8(X_b\times X_b)\ \ \
		\forall b\in B^\circ.\]  
		Applying Theorem \ref{thm:mainHilbmore} once more, we find that
		\begin{equation}\label{this}   \Bigl((\Delta_{\rX}+\Gamma_\iota)\circ
		\pi^2_{\rX} -\gamma\Bigr)\vert_{X_b\times X_b}=0\ \ \ \hbox{in}\ \CH^4(X_b\times
		X_b)\ \ \ \forall b\in B^\circ. \end{equation}    
		For general $b\in B^\circ$, the restriction $\gamma\vert_{X_b\times X_b}$
		will be supported on (divisor)$\times$(divisor), and so $\gamma\vert_{X_b\times
			X_b}$
		will act as $0$ on $\CH^2(X_b)_{(2)}$. It follows that
		\[ (\iota_b)^\ast=-\id\colon\ \ \ \CH^2(X_b)_{(2)}\ \to\ \CH^2(X_b)_{(2)}\ \
		\ \hbox{for\ general\ }b\in B^\circ.\]
		To extend this to {\em all\/} $b\in B^\circ$, one notes that the above
		construction can be done with a divisor $\rD\subset\rX$ in general position with
		respect to $X_b$.
		
		The statement for $\CH^4(X_b)_{(2)}$ follows upon taking the transpose of
		relation (\ref{this}), and using the relation (\ref{comm}). The remaining
		statements are proven similarly.
	\end{proof}
	
	\begin{rmk} Corollary \ref{cor:blochkata} was proven in a more convoluted
		way in \cite{moi}.
	\end{rmk}

	\section{Lehn-Lehn-Sorger-van Straten hyper-K\"ahler
		eightfolds}\label{sect:LLSvS}
	In this section we first show Theorem \ref{thm:mainBD2} and then deduce from it
	Theorem \ref{thm:mainLLSvS}. 
	
	Keep the same notation as in \S \ref{sect:BD}. We still have a correspondence\,:
	\begin{equation*}
	\xymatrix{
		&\rF\times_{B} \rF\ar[dr]^{q:=(p,p)}\ar[dl]_{\pi_{2}\deff (\pi, \pi)}&\\
		B&& G\times G
	}
	\end{equation*}
	However the problem is that $q$ is no longer a projective bundle\,: the fiber of
	$q$ over a pair of lines $(l, l')$ is the subspace of cubic fourfolds containing
	both $l$ and $l'$, whose dimension depends therefore on the relative position of
	$(l, l')$. To adapt the same strategy to this case, we use similar techniques as
	in \cite{MR3099982}, \cite{MR3343908} by studying the various strata of the
	morphism $q$. There are three possible relative positions between two projective
	lines in $\P^{5}$\,: identical, intersecting but not identical, not
	intersecting.
	
	On the one hand,  for a (general) cubic fourfold $X$ with Fano variety of lines
	$F$, let $$I\deff\left\{(l, l')\in F\times F~~\vert~~ l\cap l'\neq
	\vide\right\}$$ be the 6-dimensional incidence subvariety of $F\times F$. The
	incidence subvariety $I$ has two natural projections to $F$ with fiber over
	$l\in F$ the surface $S_{l}$ parameterizing lines inside $X$ meeting $l$. 
	Similarly, we consider the family version of this incidence subvariety inside
	$\rF\times _{B}\rF$\,:
	$$\rI\deff \left\{(b, l, l')\in \rF\times_{B} \rF~~\vert~~ l\cap l'\neq
	\vide\right\}=\left\{(b, l, l')\in B\times G\times G~~\vert~~ l, l'\subset
	X_{b}~~; l\cap l'\neq \vide\right\}.$$
	On the other hand, we define $J\deff\left\{(l, l')\in G\times G~~\vert~~ l\cap
	l'\neq \vide\right\}$ to be the incidence subvariety of $G\times G$.
	
	These incidence subvarieties, together with the diagonals, give the
	stratification\,:
	\begin{equation*}
	\xymatrix{
		\rF=\Delta_{\rF/B}\ar@{^{(}->}[r] \ar[d]_{p} & \rI \ar@{^{(}->}[r]
		\ar[d]^{q|_{\rI}}& \rF\times _{B}\rF \ar[d]^{q} \ar[r]^-{\pi_{2}}& B\\
		G=\Delta_{G}\ar@{^{(}->}[r] & J\ar@{^{(}->}[r] & G\times G &
	}
	\end{equation*}
	where $q$ is a projective bundle outside of $\rI$ and $q|_{\rI}$ is also a
	projective bundle outside of $\Delta_{\rF}$\,; in other words, $q$ is a
	\emph{stratified projective bundle} in the sense of Definition \ref{def:SPB}.
	
	Let $B^{\circ}$ be the Zariski open subset of $B$ parameterizing smooth cubic
	fourfolds.
	Applying Proposition \ref{prop:SPB} to $q$, we have the following analogue of
	Lemma \ref{lemma:Restriction}  and Proposition \ref{prop:Restriction} in our
	case\,:
	\begin{prop}\label{prop:RestIm}
		For any $b\in B^{\circ}$, we have
		\begin{eqnarray*}
			&&\im\left(\CH^\ast(\rF\times_{B}\rF)\to \CH^\ast(F_{b}\times F_{b})\right)\\
			&=&\im\left(\CH^\ast(G\times G)\to \CH^\ast(F_{b}\times F_{b})\right)+i_{*}\im
			\left(\CH^\ast(J)\to \CH^\ast(I_{b})\right)+\Delta_{*}\im\left(\CH^\ast(G)\to
			\CH^\ast(F_{b})\right),
		\end{eqnarray*}
		where $i: I_{b}\inj F_{b}\times F_{b}$ and $\Delta: F_{b}\inj F_{b}\times F_{b}$
		are the inclusions.
	\end{prop}
	
	As the incidence subvariety $J$ is singular along the smaller stratum
	$\Delta_{G}$, it is more convenient to work with a natural resolution of
	singularities. To this end, we define 
	\begin{eqnarray*}
		\tilde \rI&\deff &\left\{ (b, x, l, l')\in B\times \P^{5}\times G\times G
		~~\vert~~ l,l'\subset X_{b}~;~ x\in l\cap l'\right\}\,;\\
		\tilde J&\deff& \left\{ (x, l, l')\in \P^{5}\times G\times G ~~\vert~~ x\in
		l\cap l'\right\}\,;\\
		\rP &\deff& \left\{ (b, x, l)\in B\times \P^{5}\times G ~~\vert~~ l\subset
		X_{b}~;~ x\in l\right\}\,;\\
		Q&\deff& \left\{ (x, l)\in \P^{5}\times G ~~\vert~~ x\in l\right\},
	\end{eqnarray*}
	where $\tilde\rI$ (\resp $\tilde J$) admits a natural birational morphism to
	$\rI$ (\resp $J$), which contracts $\rP$ (\resp $Q$) to $\rF$ (\resp $G$). We
	summarize the situation in the following diagram whose squares are all
	cartesian\,:
	\begin{equation*}
	\xymatrix{
		\rF \ar@{->>}[d]^{p} \cart& \rP \ar@{->>}[l]\ar@{->>}[d]^{q'|_{\rP}}  
		\ar@{^{(}->}[r] \cart& \tilde \rI \ar@{->>}[d]^{q'}  \ar[r]\cart& \rI
		\ar@{->>}[d]^{q|_{\rI}} \ar@{^{(}->}[r] \cart& \rF\times_{B}\rF
		\ar@{->>}[d]^{q}\\
		G & Q \ar@{->>}[l] \ar@{^{(}->}[r]&\tilde J  \ar[r] & J \ar@{^{(}->}[r]& G\times
		G
	}
	\end{equation*}
	
	Recall that  $G=\Gr(\P^{1}, \P^{5})$,  $S$ is the tautological rank-2 subbundle,
	$g\deff c_{1}(S^{\dual}|_{F})\in \CH^{1}(F)$ is the Pl\"ucker polarization
	class, and $c\deff c_{2}(S^{\dual}|_{F})\in\CH^{2}(F)$. We computed in Lemma
	\ref{lemma:InBV} that $c_{2}(F)=5g^{2}-8c$. In $\CH^\ast(F\times F)$,
	$g_{i}\deff\pr_{i}^{*}(g)$ and $c_{i}\deff\pr_{i}^{*}(c)$ for $i=1, 2$.
	
	\begin{defi}[Tautological ring of $F\times F$]\label{def:tautFxF}
		Let $X$ be a smooth cubic fourfold and $F$ be its Fano variety of lines. We
		define the \emph{tautological ring} of $F\times F$, denoted by $R^\ast(F\times
		F)$, to be the $\Q$-subalgebra of $\CH^\ast(F\times F)$ generated by the classes
		$c_{1}, c_{2}, g_{1}, g_{2}, \Delta, I$, where $\Delta$ and $I$ are the classes
		in $\CH^\ast(F\times F)$ of the diagonal $\Delta_{F}$ and the incidence
		subvariety $I$ respectively.
	\end{defi}
	
	\begin{prop}\label{prop:AllTaut}
		For any point $b\in B^{\circ}$, we have
		\begin{equation*}
		\im\left(\CH^\ast(\rF\times_{B}\rF)\to \CH^\ast(F_{b}\times F_{b})\right) =
		R^\ast(F_{b}\times F_{b}).
		\end{equation*}
	\end{prop}
	\begin{proof}
		To simplify the notation, let us leave out the subscript $b$. Thanks to
		Proposition \ref{prop:RestIm}, we only need to deal with the following three
		cases\,:
		\begin{itemize}
			\item For $\im\left(\CH^\ast(G\times G)\to \CH^\ast(F\times F)\right)$, it is
			enough to observe that $\CH^\ast(G\times G)$ satisfies the K\"unneth formula
			(since the cycle class map $\CH^\ast(G\times G) \to H^\ast(G\times G,\Q)$ is an
			isomorphism).
			\item For $i_{*}\im \left(\CH^\ast(J)\to \CH^\ast(I)\right)$, consider $$\tilde
			I\deff \left\{ (x, l, l')\in X\times F\times F ~~\vert~~ x\in l\cap l'\right\}
			~~~\text{ and }$$ $$\tilde J\deff \left\{ (x, l, l')\in \P^{5}\times G\times G
			~~\vert~~ x\in l\cap l'\right\}$$ fitting into the diagram
			\begin{equation*}
			\xymatrix{
				\tilde I \ar@{^{(}->}[d] \ar[r]^{\tau'} \cart& I \ar@{^{(}->}[d]
				\ar@{^{(}->}[r]^-{i} \cart& F\times F \ar@{^{(}->}[d]\\
				\tilde J \ar[r]^{\tau} \ar[d]^{\pi}& J \ar@{^{(}->}[r]^-{j}& G\times G\\
				\P^{5}&&
			}
			\end{equation*}
			Denote by $\tilde i=\tau'\circ i$ and $\tilde j=\tau\circ j$. Then any cycle in
			$J$ can be written as $\tau_{*}(\alpha)$ for some $\alpha\in \CH^\ast(\tilde
			J)$. Observe that $\tilde J$ is a $\P^{4}\times \P^{4}$-bundle over $\P^{5}$
			such that the two relative $\rO(1)$ on the fibers are given by $\tilde
			j^{*}(g_{1})$ and $\tilde j^{*}(g_{2})$, respectively. Therefore $\alpha$ is a
			linear combination of cycles of the form $\pi^{*}(h^{k})\tilde
			j^{*}(g_{1}^{l}g_{2}^{m})$ where $k, l,m\in\N$ and $h=\rO_{\P^{5}}(1)$. We have
			\begin{eqnarray*}
				&&i_{*}(\tau_{*}(\pi^{*}(h^{k})\tilde j^{*}(g_{1}^{l}g_{2}^{m}))|_{I})\\
				&=&i_{*}\circ \tau'_{*}\left(\pi^{*}(h^{k})\tilde
				j^{*}(g_{1}^{l}g_{2}^{m})|_{\tilde I}\right)\\
				&=& \tilde i_{*}\left(\pi^{*}(h^{k})|_{\tilde I}\cdot \tilde
				i^{*}(g_{1}^{l}g_{2}^{m})\right)\\
				&=& g_{1}^{l}g_{2}^{m}\cdot i_{*}(\tau_{*}\pi^{*}(h^{k})|_{I})\\
				&=& g_{1}^{l}g_{2}^{m}\cdot \Gamma_{h^{k}},
			\end{eqnarray*}
			where $\Gamma_{h^{k}}$, defined in \cite[Appendix A]{SV}, is the cycle of
			$F\times F$ represented by the subvariety $$\left\{(l, l')\in F\times F
			~~\vert~~ \exists x\in H_{1}\cap \cdots \cap H_{k} \text{  such that } x\in
			l\cap l'\right\},$$ where $H_{1}, \cdots, H_{k}$ are $k$ general hyperplanes in
			$\P^{5}$. It is proven in \cite[Appendix A]{SV} that when $k\geq 1$,
			$\Gamma_{h^{k}}$ is actually a polynomial in $c_{1}, c_{2}, g_{1}, g_{2}$, while
			$\Gamma_{h^{0}}=I$.
			\item For $\Delta_{*}\im\left(\CH^\ast(G)\to \CH^\ast(F)\right)$, let us remark
			that for any $\alpha\in \CH^\ast(F)$, we have
			$\Delta_{*}(\alpha)=\Delta\cdot\pr_{1}^{*}(\alpha)$. Thus it suffices to recall
			that $\im\left(\CH^\ast(G)\to \CH^\ast(F)\right)$ is generated by $g$ and $c$.
		\end{itemize}
	\end{proof}
	
	Consequently, in order to prove Theorem \ref{thm:mainBD2}, we need to study the
	injectivity of the cycle class map restricted to the tautological ring
	$R^\ast(F\times F)$.

	\begin{prop}\label{prop:InjTaut}
		Let $X$ be a smooth cubic fourfold and let $F$ be its Fano variety of lines.
		Then the cycle class map restricted to the tautological ring $R^{*}(F\times F)$
		is injective.
	\end{prop}
	\begin{proof}
		It suffices to show the proposition for general cubic fourfolds, in which case
		$$\cl: R^\ast(F\times F)\to Hdg^{2*}(F\times F)_{\Q}$$ is surjective. Let us
		show it is injective.\\
		First it is not hard to count the dimensions of the spaces of Hodge classes\,:
		
		\begin{tabular}{c|c|c|c|c|c|c|c|c|c}
			$i$& 0&1&2&3&4&5&6&7&8\\
			\hline
			$\dim Hdg^{2i}$ & 1&2&6&8&12&8&6&2&1
		\end{tabular}

		\noindent It is enough to show that the $R^{i}(F\times F)$ have the same
		dimensions.
		
		The following relations in $R^{*}(F\times F)$ are at our disposal.
		\begin{enumerate}[(i)]
			\item $g_1\cdot \Delta=g_{2}\cdot \Delta\,; c_{1}\cdot\Delta=c_{2}\cdot\Delta$.
			\item For $i=1, 2$, we have $12g_{i}c_{i}=5g_{i}^{3}\,; 4\,
			c_{i}^{2}=g_{i}^{4}$.
			\item Voisin's relation \cite{MR2435839}\footnote{The coefficients are made
				precise by \cite[Proposition 17.4]{SV}.}\,:
			$$I^{2}=2\Delta+I\cdot(g_{1}^{2}+g_{1}g_{2}+g_{2}^{2})+\Gamma_{2}(g_{1},
			g_{2},c_{1},c_{2}),$$ where $\Gamma_{2}$ is a polynomial of weighted degree 4.
			\item In \cite[Proposition 17.5]{SV}, one finds $$\Delta\cdot
			I=6c_{1}\Delta-3g_{1}^{2}\Delta.$$
			\item In \cite[Lemma 17.6]{SV}, there is a polynomial $P$ of weighted degree 4
			such that
			$$c_{1}\cdot I=P(g_{1}, g_{2},c_{1}, c_{2})\,;$$
			$$c_{2}\cdot I=P(g_{2}, g_{1},c_{2}, c_{1}).$$
		\end{enumerate}
		
		Using these relations, we easily get for each degree a list of generators (as
		vector-spaces)\,:
		\begin{itemize}
			\item $R^{0}=\langle \1 \rangle$\,;
			\item $R^{1}=\langle g_{1}, g_{2} \rangle$\,;
			\item $R^{2}=\langle g_{1}^{2}, g_{1}g_{2}, g_{2}^{2}, c_{1}, c_{2},
			I\rangle$\,;
			\item $R^{3}=\langle g_{1}^{3}, g_{1}^{2}g_{2}, g_{1}g_{2}^{2}, g_{2}^{3},
			g_{1}c_{2}, g_{2}c_{1}, g_{1}I, g_{2}I \rangle$\,;
			\item $R^{4}=\langle g_{1}^{4}, g_{1}^{3}g_{2}, g_{1}^{2}g_{2}^{2},
			g_{1}g_{2}^{3}, g_{2}^{4}, g_{1}^{2}c_{2}, g_{2}^{2}c_{1}, c_{1}c_{2},
			g_{1}^{2}I, g_{2}^{2}I, g_{1}g_{2}I, \Delta \rangle$\,;
			\item $R^{5}=\langle g_{1}^{4}g_{2}, g_{1}^{3}g_{2}^{2}, g_{1}^{2}g_{2}^{3},
			g_{1}g_{2}^{4}, g_{1}^{3}c_{2}, g_{2}^{3}c_{1}, g_{1}^{2}g_{2}I,
			g_{1}g_{2}^{2}I, g_{1}\Delta\rangle$\,;
			\item $R^{6}=\langle g_{1}^{4}g_{2}^{2}, g_{1}^{3}g_{2}^{3}, g_{1}^{2}g_{2}^{4},
			g_{1}^{4}c_{2}, g_{2}^{4}c_{1}, g_{1}^{2}g_{2}^{2}I, g_{1}^{2}\Delta 
			\rangle$\,;
			\item $R^{7}=\langle g_{1}^{4}g_{2}^{3}, g_{1}^{3}g_{2}^{4} \rangle$\,;
			\item $R^{8}=\langle g_{1}^{4}g_{2}^{4} \rangle$.
		\end{itemize}
		Observe that we have the same number of generators as the dimension of
		$Hdg^{2i}$ for $i\neq 5 \text{ or } 6$.
		Therefore the cycle class map $R^{i}(F\times F)\to H^{2i}(F\times F, \Q)$ is
		injective for $i=0,1,2,3,4,7,8$.
		\begin{enumerate}
			\item[(vi)] As for $i=5$ (\resp $i=6$), we use the following (new) tautological
			relation  established in the Appendix Theorem \ref{thm:Relation}\,:
			$$6\Delta_{*}(g)+ g_{1}g_{2}(g_{1}+g_{2})\cdot I= Q(g_{1}, g_{2}, c_{1},
			c_{2}),
			$$
			where $Q$ is a polynomial. 
		\end{enumerate}
		Therefore the generator $g_{1}\Delta=\Delta_{*}(g)$ (\resp $g_{1}^{2}\Delta$) is
		redundant, hence $R^{i}(F\times F)\to H^{2i}(F\times F)$ is also injective in
		these two degrees.
	\end{proof}
	
	\begin{rmk}
		As a manifestation of the same principle as in \S \ref{subsect:beyond}, the
		extra difficulty encountered here (excess dimension of $I$, the new tautological
		relation \emph{etc.}) can be traced back to the lack of positivity of the vector
		bundle $E=\Sym^{3}S^{\vee}$ on $G=\Gr(\P^{1}, \P^{5})$, namely it satisfies only
		$(\star_{1})$ but not $(\star_{2})$, where $S$ is the tautological subbundle on
		$G$.
	\end{rmk}

	We can now easily conclude the proof of Theorem \ref{thm:mainBD2}\,:
	\begin{proof}[Proof of Theorem \ref{thm:mainBD2}]
		As the standard conjecture is proven for $F_{b}$ in \cite{MR3040747} (this can
		also be seen more elementarily by noting that the incidence correspondence $I$
		induces an isomorphism from $H^6(F_b,\Q)$ to $H^2(F_b,\Q)$), numerical
		equivalence coincides with homological equivalence on powers of $F_b$. Since the
		moduli stack $\mathcal{C}$ is dominated by the parameter space $B^{\circ}$ of
		smooth cubic fourfolds, by Remark \ref{rmk:Parameter}, we only need to show the
		conclusion for the family $\rF^{\circ}\times_{B^{\circ}}\rF^{\circ}\to
		B^{\circ}$. Since any cycle of $\rF^{\circ}\times_{B^{\circ}}\rF^{\circ}$ is the
		restriction of a cycle of $\rF\times_{B}\rF$,  it is enough to show that for any
		$b\in B^{\circ}$, the restriction of a cycle $\gamma\in
		\CH^\ast(\rF\times_{B}\rF)$ to $F_{b}\times F_{b}$ is zero if and only if it is
		homologically trivial, which is proven by combining Proposition
		\ref{prop:AllTaut} and Proposition \ref{prop:InjTaut}.
	\end{proof}
	
	With Theorem \ref{thm:mainBD2} proven, we proceed to study the $0$-cycles and
	codimension-$2$ cycles of the LLSvS hyper-K\"ahler eightfolds. The key input is
	Voisin's degree $6$ dominant rational map \cite[Proposition 4.8]{MR3524175}
	$$F\times F\dasharrow Z.$$
	Let $B^{\circ\circ}$ be the Zariski open subset of $B$ parameterizing smooth
	cubic fourfolds not containing a plane. 
	Consider the family version of Voisin's construction (over $B^{\circ\circ}$)\,:
	$\psi: \rF^{\circ\circ}\times_{B^{\circ\circ}}\rF^{\circ\circ}\dasharrow \rZ$.
	\begin{proof}[Proof of Theorem \ref{thm:mainLLSvS}]
		Take a resolution of indeterminacies\,:
		\begin{equation*}
		\xymatrix{
			\tilde{\rF^{\circ\circ}\times_{B^{\circ\circ}}\rF^{\circ\circ}} \ar[d]_{\tau}
			\ar[dr]^{f} & \\
			\rF^{\circ\circ}\times_{B^{\circ\circ}}\rF^{\circ\circ}  \ar@{-->}[r]_-{\psi}&
			\rZ.
		}
		\end{equation*}
		For $(i)$, let $\gamma\in\CH^{8}(\rZ)_{}$ be a relative $0$-cycle whose degree
		on fibers is zero. Then, for any $b\in B^{\circ\circ}$,
		$$\left(\tau_{*}f^{*}(\gamma)\right)|_{F_{b}\times
			F_{b}}=\tau_{b*}\left(f^{*}(\gamma)|_{\tilde{F_{b}\times
				F_{b}}}\right)=\tau_{b*}f_{b}^{*}\left(\gamma|_{Z_{b}}\right).$$ Thus
		$\tau_{*}f^{*}(\gamma)$ is a relative $0$-cycle of fiber degree zero on
		$\rF^{\circ\circ}\times_{B^{\circ\circ}}\rF^{\circ\circ}$ and by Theorem
		\ref{thm:mainBD2}, we know that
		$$\tau_{b*}f_{b}^{*}\left(\gamma|_{Z_{b}}\right)=0 \text{ in }
		\CH^{8}(F_{b}\times F_{b}).$$
		For $b\in B^{\circ\circ}$ general, $\tau_{b}$ is birational hence induces an
		isomorphism on $\CH_{0}$, hence $f_{b}^{*}\left(\gamma|_{Z_{b}}\right)=0$.
		Moreover, since $f_{b}$ is generically finite of degree 6 (still under the
		assumption that $b$ is general), we have
		$$\gamma|_{Z_{b}}=\frac{1}{6}f_{b*}f_{b}^{*}(\gamma|_{Z_{b}})=0.$$
		A specialization argument shows that $\gamma|_{Z_{b}}=0$ for all $b\in B^{\circ
			\circ}$.\\
		As for $(ii)$, \ie, codimension-$2$ cycles\,: since $H^{3}(Z_{b},
		\Q)=H^{3}(F_{b}\times F_{b})=0$, any cycle in $\CH^{2}(Z_{b})$ or in
		$\CH^{2}(F_{b}\times F_{b})$ is homologically trivial if and only if its
		Abel--Jacobi invariant vanishes. Now the same proof as in $(i)$ works because
		the Abel--Jacobi kernel for codimension-$2$ cycles $\CH^{2}_{AJ}$, just as
		$\CH_{0}$, is a birational invariant (for smooth projective varieties), hence 
		$$\tau_{b*}: \CH^{2}(\tilde{F_{b}\times F_{b}})_{hom}\to \CH^{2}(F_{b}\times
		F_{b})_{hom}$$ 
		is an isomorphism.
	\end{proof}
	
	\begin{proof}[Proof of Corollary \ref{cor:BVChern}]
		In view of Theorem \ref{thm:mainLLSvS}, this is just a special case of
		Proposition \ref{prop:ChernBV}.
	\end{proof}
	
	\begin{rmk} As above, let $Y_b$ be a smooth cubic fourfold not containing a
		plane, and $Z_b$ the associated LLSvS eightfold.
		Our argument to prove Theorem \ref{thm:mainLLSvS} breaks down for $\CH^j(Z_b)$
		with $2<j<8$, because Voisin's map is not a morphism. It is known \cite{Mura},
		\cite{Chen} that the indeterminacy locus of Voisin's map is the incidence
		subvariety $I\subset F_b\times F_b$, and a resolution of indeterminacy is
		obtained by blowing up $I$. To extend Theorem~\ref{thm:mainLLSvS} to the full
		Chow ring $\CH^\ast(Z_b)$, it remains to prove analogues of Propositions
		\ref{prop:AllTaut} and \ref{prop:InjTaut} for $\mathcal{I}$, the family of
		incidence varieties. 
	\end{rmk}

	\newpage
	\appendix
	
	\section{On a new tautological relation}
	\begin{center}
		\footnotesize{{\sc Lie Fu, Robert Laterveer, Mingmin Shen and Charles Vial}}
	\end{center}
	
	Let $X$ be a smooth cubic fourfold and $F$ be its Fano variety of lines, which
	is a hyper-K\"ahler fourfold by \cite{MR818549}. In this appendix, we establish
	a new relation (Theorem \ref{thm:Relation}), up to rational equivalence, among
	3-dimensional tautological cycle classes of $F\times F$. Some interesting
	applications of this tautological relation are also discussed. We try to keep
	the appendix as self-contained as possible.

	Throughout this appendix, let us fix the following notation\,:
	
	\begin{itemize}
		\item $\P^{5}$ is the ambient space and $X$ is a smooth cubic hypersurface in
		it.
		\item $h:=c_{1}(\rO_{\P^{5}}(1))$\,; $h|_{X}$ is still denoted by $h$.
		\item $G:=\Gr(\P^{1}, \P^{5})\isom \Gr(2,6)$ is the Grassmannian of projective
		lines in $\P^{5}$.
		\item $F:=F(X)$ is the Fano variety of lines of $X$.
		\item $S$ is the tautological subbundle on $G$.
		\item $g:=c_{1}(S^{\vee})$ is the Pl\"ucker polarization class\,; $g|_{F}$ is
		still denoted by $g$.
		\item $c:=c_{2}(S)$\,; $c|_{F}$ is still denoted by $c$.
		\item $h_{i}:=\pr_{i}^{*}(h)$, $g_{i}:=\pr_{i}^{*}(g)$ and
		$c_{i}:=\pr_{i}^{*}(c)$ where $\pr_{i}$ is the $i$-th projection.
		\item If $P:=\P(S|_{F})$ denotes the incidence variety in $F\times X$, then the
		natural projection $p: P\to F$ is  the universal projective line and $q: P\to X$
		is the evaluation map.
		\item $I\subset F\times F$ is the incidence subvariety parameterizing pairs of
		intersecting lines contained in~$X$.
		\item $\tilde I:=P\times_{X} P$. Note that $I$ is its image in $F\times F$ via
		the natural projection.
	\end{itemize}

	The main result of this appendix is the following.
	\begin{thm}\label{thm:Relation} 
		There exists a polynomial $Q$ (of weighted degree 5) such that the following
		equality holds in $\CH^{5}(F\times F)$\,:
		\begin{equation}\label{eq:relation}
		6\Delta_{*}(g)+ g_{1}g_{2}(g_{1}+g_{2})\cdot I= Q(g_{1}, g_{2}, c_{1}, c_{2}),
		\end{equation}
		where $\Delta: F\inj F\times F$ is the diagonal embedding.
	\end{thm}
	
	\begin{rmk}
		The polynomial $Q$ is not unique. A cohomological computation shows that
		\[
		Q(g_{1}, g_{2}, c_{1}, c_{2})= \frac{1}{4}(g_1^4g_2 +g_1g_2^4)
		+\frac{7}{12}(g_1^3g_2^2 + g_1^2g_2^3)
		\]
		is one possible choice of $Q$.
	\end{rmk}

	\subsection{Proof of the tautological relation}
	
	We have the following diagram
	\begin{equation}\label{diag:Appendix}
	\xymatrix{
		\tilde I  \ar[r]  \ar@{^{(}->}[d]^{i} & X \ar@{^{(}->}[d]^{\Delta_{X}} \\
		P\times P  \ar[r]^{(q,q)} \ar[d]^{(p,p)} & X\times X \\
		F\times F &\\
	}
	\end{equation}
	Let us first introduce some natural cycles on $F\times F$. For any $i\in \N$,
	define
	$$\Gamma_{h^{i}}:=(p,p)_{*}(q,q)^{*}({\Delta_{X}}_{*}(h^{i}))\in
	\CH^{i+2}(F\times F).$$
	Note that $\Gamma_{h^{0}}$ is nothing but the incidence correspondence $I$.
	Geometrically, $\Gamma_{h^{i}}$ is represented by the locus of pairs of lines
	contained in $X$ intersecting at a point which lies on the intersection of $i$
	general hyperplane sections of $X$.
	
	\begin{lemma}\label{lemma:Gammah}
		For any $i>0$, the cycle $\Gamma_{h^{i}}$ is a polynomial of $g_{1}, g_{2},
		c_{1}, c_{2}$. Precisely,
		\begin{eqnarray*}
			\Gamma_{h}&=&\frac{1}{18}(g_{1}^{3}+6g_{1}^{2}g_{2}+6g_{1}g_{2}^{2}+g_{2}^{3}-6g_{1}c_{2}-6g_{2}c_{1})\,;\\
			\Gamma_{h^{2}}&=& \frac{1}{18}(g_{1}^{3}g_{2} +6 g_{1}^{2}g_{2}^{2}+
			g_{1}g_{2}^{3} -6g_{1}^{2}c_{2}-6g_{2}^{2}c_{1}+ 6c_{1}c_{2})\,;\\
			\Gamma_{h^{3}}&=&\frac{1}{18}(g_{1}^{3}g_{2}^{2}+g_{1}^{2}g_{2}^{3}-g_{1}^{3}c_{2}-g_{2}^{3}c_{1})\,;\\
			\Gamma_{h^{4}}&=&\frac{1}{108}g_{1}^{3}g_{2}^{3}.
		\end{eqnarray*}
	\end{lemma}
	
	\begin{proof}
		A slightly more complicated (but equivalent) form of the first two formulas is
		proven in \cite[Proposition A.6]{SV}. For the convenience of the reader, we give
		a complete proof here. The excess intersection formula \cite[\S 6.3]{MR1644323}
		applied to the following cartesian diagram 
		\begin{displaymath}
		\xymatrix{
			X  \ar[r]  \ar[d]^{\Delta_{X}} \cart& \P^{5} \ar[d]^{\Delta_{\P}}\\
			X\times X \ar[r] & \P^{5}\times \P^{5}
		}
		\end{displaymath}
		yields that, for any $i \in \N$, we have in $\CH^\ast(X\times X)$
		$$3{\Delta_{X}}_{*}(h^{i+1})={\Delta_{\P}}_{*}(h^{i})|_{X\times X}.$$
		From ${\Delta_{\P}}_{*}(h^{i})=h_{1}^{5}h_{2}^{i}+\cdots +h_{1}^{i}h_{2}^{5}$,
		we obtain
		$${\Delta_{X}}_{*}(h^{i})=\frac{1}{3}\left(h_{1}^{4}h_{2}^{i}+\cdots
		+h_{1}^{i}h_{2}^{4}\right).$$
		Therefore 
		\begin{eqnarray*}
			\Gamma_{h^{i}}&=&(p,p)_{*}(q,q)^{*}(\Delta_{X,*}(h^{i}))\\
			&=&\frac{1}{3}(p,p)_{*}(q,q)^{*}\left(h_{1}^{4}h_{2}^{i}+\cdots
			+h_{1}^{i}h_{2}^{4}\right)\\
			&=&\frac{1}{3}\left(f_{4}\times f_{i}+\cdots+f_{i}\times f_{4}\right),
		\end{eqnarray*}
		where $f_{j}:=p_{*}q^{*}(h^{j})$ and where $\times$ is the exterior product
		$\pr_{1}^{*}(-)\cdot \pr_{2}^{*}(-)$. All the formulas in the statement then
		follow from the  facts that $f_{1}=1$, $f_{2}=g$, $f_{3}=g^{2}-c$ and
		$f_{4}=\frac{1}{6}g^{3}$ (\cf \cite[Lemma A.4]{SV}, \cite[Lemma 3.2]{MR2435839}
		and \cite[Lemma 3.5]{MR2435839}).
	\end{proof}
	
	Define $I_{0}:=I\backslash \Delta_{F}$ to be the subvariety of $F\times F$
	parameterizing pairs of distinct intersecting lines in~$X$. We then have a
	natural morphism 
	$$q_{0}: I_{0}\to X$$
	which sends two lines to their intersection point. 
	
	\begin{lemma}\label{lemma:normalbundle}
		The inclusion $I_{0}\inj F\times F\backslash\Delta_{F}$ is a local complete
		intersection and the Chern classes of the normal bundle $N:=N_{I_{0}/{F\times
				F\backslash\Delta_{F}}}$ are given by
		\begin{eqnarray*}
			c_{1}(N)&=&(g_{1}+g_{2})|_{I_{0}}-q_{0}^{*}(h)\,;\\
			c_{2}(N)&=&(g_{1}^{2}+g_{1}g_{2}+g_{2}^{2})|_{I_{0}}-3(g_{1}+g_{2})|_{I_{0}}\cdot
			q_{0}^{*}(h)+6q_{0}^{*}(h^{2}).
		\end{eqnarray*}
	\end{lemma}
	\begin{proof}
		Note that $\tilde{I}\subset P\times P$ is a local complete intersection (since
		$\tilde{I}\subset P\times P$ is obtained from the local complete intersection
		$\Delta_X\subset X\times X$ via base change) and that $\tilde{I}\backslash
		\Delta_P \subset P\times P$ is a section of $P\times P\to F\times F$ over $I_0$.
		We apply \cite[B.7.5]{MR1644323} and see that $I_0\subset F\times F\backslash
		\Delta_F$ is a local complete intersection. Using the section
		$\tilde{I}\backslash \Delta_P$, we view $I_0$ as a subvariety of $P\times P$.
		Then we get the following short exact sequence
		\[ 0 \to  pr_1^*T_{P/F}\oplus pr_2^*T_{P/F} \to N_{I_0/P\times P} \to
		N_{I_0/F\times F}\to 0 \ .\]
		Note that by construction, we have
		\[
		N_{I_0/P\times P} = q_0^*T_X.
		\]
		The Chern classes of $N_{I_0/F\times F}$ are computed as follows:
		\begin{align*}
		c(N) & = \frac{q_0^*c(T_X)}{pr_1^* c(T_{P/F})\cdot pr_2^* c(T_{P/F})}\\
		&= \frac{(1+h)^6}{(1+3h)(1+2 q_0^* h - g_1|_{I_0})(1+2 q_0^* h - g_2|_{I_0})}.
		\end{align*}
		The lemma follows from the expansion of the above equation.
	\end{proof}
	
	\begin{rmk}
		The previous lemma implies that 
		$$I^{2}|_{F\times F\backslash \Delta_{F}}=I\cdot
		(g_{1}^{2}+g_{1}g_{2}+g_{2}^{2})-3(g_{1}+g_{2})\Gamma_{h}+6\Gamma_{h^{2}}.$$
		Thus by Lemma \ref{lemma:Gammah} there exists $\alpha\in \Q$ and a polynomial
		$\Gamma_{2}$ such that in $\CH^{4}(F\times F)$ we have 
		$$I^{2}=\alpha\cdot
		\Delta_{F}+I\cdot(g_{1}^{2}+g_{1}g_{2}+g_{2}^{2})+\Gamma_{2}(g_{1}, g_{2},
		c_{1}, c_{2}),$$  for some $\alpha\in \Q$.
		This was proven by Voisin \cite{MR2435839}. In fact, $\alpha=2$, as is computed
		in \cite[Proposition 17.4]{SV}.
	\end{rmk}
	
	\begin{proof}[Proof of Theorem \ref{thm:Relation}]
		Let us first prove the theorem for a general cubic fourfold $X$. Fix three
		general hyperplane sections $H_{1}, H_{2}, H_{3}$ of $X$. For $i=1,2,3$, let 
		$$Z_{i}:=\left\{(l,l')\in F\times F~~\vert~~ l\cap l'\cap H_{1}\cap\cdots\cap
		H_{i}\neq \vide\right\}.$$
		On the one hand, as mentioned before, the class of $Z_{i}$ in $\CH^{2+i}(F\times
		F)$ is equal to $\Gamma_{h^{i}}$\,; on the other hand, denoting
		$Z_{i}^{o}:=Z_{i}\backslash \Delta_{F}$ the complement of the diagonal in
		$Z_{i}$, the class of $Z^{o}_{i}$ in $\CH^{i}(I_{0})$ is equal to
		$q_{0}^{*}(h^{i})$ by definition. This yields the diagram 
		
		\begin{displaymath}
		\xymatrix{
			Z^{o}_{3}\subset Z^{o}_{2}\subset Z_{1}^{o}\ar@{^{(}->}[r] & I_{0}
			\ar[d]^{q_{0}}\ar@{^{(}->}[r]^-{\iota} & F\times F\backslash \Delta_{F}\\
			&X&
		}
		\end{displaymath}
		Denoting $N$ the normal bundle of $\iota$, we obtain
		\begin{eqnarray*}
			I\cdot \Gamma_{h}|_{F\times F\backslash \Delta_{F}} &=& I_{0}\cdot
			\iota_{*}q_{0}^{*}\left(h\right)\\
			&=& \iota_{*}\left(q_{0}^{*}(h)\cdot c_{2}(N)\right)\\
			&=& \iota_{*}\left((g_{1}^{2}+g_{1}g_{2}+g_{2}^{2})|_{I_{0}}\cdot
			q_{0}^{*}(h)-3(g_{1}+g_{2})|_{I_{0}}\cdot
			q_{0}^{*}(h^{2})+6q_{0}^{*}(h^{3})\right)\\
			&=& \left((g_{1}^{2}+g_{1}g_{2}+g_{2}^{2})\cdot Z_{1}- 3(g_{1}+g_{2})\cdot
			Z_{2}+6 Z_{3}\right)|_{F\times F\backslash \Delta_{F}}\\
			&=& \left((g_{1}^{2}+g_{1}g_{2}+g_{2}^{2})\cdot \Gamma_{h}- 3(g_{1}+g_{2})\cdot
			\Gamma_{h^{2}}+6 \Gamma_{h^{3}}\right)|_{F\times F\backslash \Delta_{F}},
		\end{eqnarray*}
		where the third equality uses Lemma \ref{lemma:normalbundle}. By Lemma
		\ref{lemma:Gammah}, there exists a polynomial $P_{1}$ such that 
		$$I\cdot \Gamma_{h}|_{F\times F\backslash \Delta_{F}}=P_{1}(g_{1}, g_{2}, c_{1},
		c_{2})|_{F\times F\backslash \Delta_{F}}.$$ 
		Here, more precisely, one can compute by Lemma \ref{lemma:Gammah} and the
		relation $12gc=5g^{3}$ that $$P_{1}(g_{1}, g_{2}, c_{1},
		c_{2})=\frac{5}{12}\left(g_{1}^{4}g_{2}+4g_{1}^{3}g_{2}^{2}+4g_{1}^{2}g_{2}^{3}+g_{1}g_{2}^{4}-3g_{1}^{3}c_{2}-3g_{2}^{3}c_{1}\right).$$
		By the localization short exact sequence of Chow groups, there exists an element
		$D\in \CH^{1}(F)$ such that in $\CH^{5}(F\times F)$ we have 
		\begin{equation*}
		I\cdot \Gamma_{h}+ \Delta_{*}(D)=P_{1}(g_{1}, g_{2}, c_{1}, c_{2}).
		\end{equation*}
		Since $X$ is assumed (for now) to be general, $\CH^{1}(F)$ is generated by $g$,
		hence $D=\lambda g$ for some $\lambda\in \Q$. This yields that in
		$\CH^{5}(F\times F)$ we have
		\begin{equation}\label{eqn:IGammah1}
		I\cdot \Gamma_{h}+\lambda \Delta_{*}(g)=P_{1}(g_{1}, g_{2}, c_{1}, c_{2}).
		\end{equation}
		However, we know that $I\cdot c_{1}$, $I\cdot c_{2}$, $I\cdot g_{1}^{3}$ and
		$I\cdot g_{2}^{3}$ are polynomials in $g_{1}, g_{2}, c_{1}, c_{2}$ by
		\cite[Lemma~17.6]{SV} (\cf the known relations collected in the proof of
		Proposition \ref{prop:InjTaut}). The first formula in  Lemma \ref{lemma:Gammah}
		then yields that 
		\begin{equation}\label{eqn:IGammah2}
		I\cdot \Gamma_{h}=\frac{1}{3}I\cdot
		\left(g_{1}^{2}g_2+g_{1}g_{2}^{2}\right)+P_{2}(g_{1}, g_{2}, c_{1}, c_{2}) 
		\end{equation}
		for some polynomial $P_{2}$.
		
		Putting (\ref{eqn:IGammah1}) and (\ref{eqn:IGammah2}) together, we know that
		there exists a polynomial $Q$ such that the following equality holds in
		$\CH^{5}(F\times F)$\,:
		$$3\lambda\cdot \Delta_{*}(g)+I\cdot
		\left(g_{1}^{2}g_2+g_{1}g_{2}^{2}\right)=Q(g_{1}, g_{2}, c_{1}, c_{2}).$$
		By considering the action of both sides on the cohomology, we easily see that
		$\lambda=2$ and that
		\[
		Q(g_1,g_2,c_1,c_2) = \frac{1}{4}(g_1^4g_2 +g_1g_2^4) +\frac{7}{12}(g_1^3g_2^2 +
		g_1^2g_2^3).
		\] 
		Therefore the desired relation is proven for a general cubic fourfold. As all
		the cycles appearing are universally defined in the universal  Fano variety of
		lines, a specialization argument shows that this relation must also hold for any
		smooth cubic fourfold.
	\end{proof}

	\subsection{Some applications to the Fourier decomposition of $F$}
	\label{sec:appfourier}

	Our aim is to use Theorem \ref{thm:mainBD2}, which is based on Theorem
	\ref{thm:Relation}, to complement the results of \cite{SV} concerning the
	multiplicative structure of the Chow motive of the Fano variety of lines on a
	smooth cubic fourfold.

	\subsubsection{An explicit Chow--K\"unneth decomposition for $F$} 
	Recall that a Chow--K\"unneth decomposition for a smooth projective variety $X$
	of dimension $d$ is a decomposition of the diagonal $\Delta_X\in \CH^d(X\times
	X)$ into a sum $\Delta_X = \pi^0_X+\cdots +\pi^{2d}_X$ of mutually orthogonal
	idempotent correspondences $\pi_X^i \in \CH^d(X\times X)$ whose action in
	cohomology is given by $(\pi_X^i)_*H^*(X,\Q) = H^i(X,\Q)$. It is a conjecture of
	Murre that all smooth projective varieties should admit a Chow--K\"unneth
	decomposition. In \cite{SV}, it is shown that the Fano variety of lines on a
	smooth cubic fourfold admits a Chow--K\"unneth decomposition\,; see especially
	\cite[Theorem~3.3]{SV}. Such a decomposition is obtained by modifying the
	following correspondences in $\CH^4(F\times F)$\,:
	\begin{equation}\label{eq:CK-F}
	\pi_F^0 = \frac{1}{23\cdot 25}{l_1^2}, \ \ \pi_F^2  = \frac{1}{25}L\cdot l_1, \
	\ \pi_F^4 = \frac{1}{2} (L^2 - \frac{1}{25}l_1\cdot l_2),\ \ \pi_F^6  =
	\frac{1}{25}L\cdot l_2,\ \ \pi_F^8 = \frac{1}{23\cdot 25}{l_2^2}.
	\end{equation}
	Here, $L := \frac{1}{3}(g_1^2 + \frac{3}{2}g_1g_2 + g_2^2 - c_1 -c_2) - I \in
	\CH^2(F\times F)$ is a (and in fact ``the'', by Proposition~\ref{prop:InjTaut})
	tautological cycle representing the Beauville--Bogomolov form\,; see
	\cite[Proposition~19.1]{SV}. The cycle $l \in \CH^2(F)$ is the restriction of
	$L$ to the diagonal, and, as before, a subscript $i$ indicates the pull-back
	along the projection $F \times F \to F$ to the $i$-th factor.
	
	As was expected from \cite[Conjecture~3]{SV}, these correspondences already
	define a Chow--K\"unneth decomposition\,:
	\begin{prop}\label{prop:CK-F}
		The correspondences in \eqref{eq:CK-F} define a Chow--K\"unneth decomposition of
		$F$.
	\end{prop}
	\begin{proof}
		The correspondences $\pi_F^{2i}$ of \eqref{eq:CK-F} are cycles on $F\times F$
		that belong to the image of the restriction map $\CH^\ast(\rF\times_{B}\rF)\to
		\CH^\ast(F\times F)$, and they define a K\"unneth decomposition of the diagonal
		in cohomology by \cite[Corollary~1.7]{SV}. (Here $\rF\to B$ is the universal
		Fano variety of lines as defined in \S \ref{sect:BD}). It follows readily from
		Theorem \ref{thm:mainBD2} that they define a Chow--K\"unneth decomposition.
	\end{proof}

	\subsubsection{A new multiplicativity statement}
	Using the Chow--K\"unneth decomposition \eqref{eq:CK-F} given by Proposition
	\ref{prop:CK-F}, we can define, for all integers $i$ and $j$, \[
	\CH^i(F)_{(j)}:=(\pi_{F}^{2i-j})_\ast \CH^i(F)\ .\] 
	Concretely, we have (\emph{cf.} \cite{SV})
	\begin{align*}
	\CH^4(F) &= \CH^4(F)_{(0)} \oplus \CH^4(F)_{(2)} \oplus \CH^4(F)_{(4)}\\
	\CH^3(F) &= \CH^3(F)_{(0)} \oplus \CH^3(F)_{(2)}\\
	\CH^2(F) &= \CH^2(F)_{(0)} \oplus \CH^2(F)_{(2)}\\
	\CH^1(F) &= \CH^1(F)_{(0)}\\
	\CH^0(F) &= \CH^0(F)_{(0)}.
	\end{align*}
	In \cite{SV}, it was proven that for the Fano variety of lines on a \emph{very
		general} cubic fourfold, the decomposition $\CH^i(F)_{(j)}$ defines a bigrading
	on the Chow ring $\CH^\ast(F)$, in the sense that for all integers $i,i',j,j'$
	we have
	$$\CH^i(F)_{(j)} \cdot \CH^{i'}(F)_{(j')} \subseteq \CH^{i+i'}(F)_{(j+j')}.$$
	In the case of the Fano variety of lines on a non-very general cubic fourfold,
	the following two relations could not be established (see \cite[Remark
	22.9]{SV})\,:
	\begin{equation}\label{eq:mult}
	\CH^1(F) \cdot\CH^2(F)_{(0)} \subseteq \CH^3(F)_{(0)}\,;
	\end{equation}
	\begin{equation}\label{eq:mult2}
	\CH^2(F)_{(0)} \cdot\CH^2(F)_{(0)} \subseteq \CH^4(F)_{(0)} = \Q \cdot
	\mathfrak{o}_F.
	\end{equation}
	Using Theorem \ref{thm:mainBD2}, which is based on the new relation
	\eqref{eq:relation}, we can now prove one of the missing two inclusions\,:
	
	\begin{prop}\label{P:multF}
		Let $F$ be the Fano variety of lines on a smooth cubic fourfold. Then
		$$\CH^1(F)\cdot \CH^2(F)_{(0)}=\CH^3(F)_{(0)}.$$
	\end{prop}
	
	\begin{proof}
		We first show that $\CH^3(F)_{(0)}\subseteq \CH^1(F)\cdot \CH^2(F)_{(0)}$. On
		the one hand, the cycle class map $\CH^{3}(F)_{(0)}\to H^{6}_{alg}(F, \Q)$ is an
		isomorphism\,; on the other hand, the hard Lefschetz isomorphism implies that
		$H^{6}_{alg}(F, \Q)$ is generated by $g^{2}\cdot H^{2}_{alg}(F, \Q)=g^{2}\cdot
		\CH^{1}(F)$. Hence $\CH^{3}(F)_{(0)}$ is generated by intersections of three
		divisors, which is contained in $\CH^{1}(F)\cdot \CH^{2}(F)_{(0)}$ since we know
		that $\CH^{1}(F)\cdot \CH^{1}(F)\subseteq \CH^{2}(F)_{(0)}$.\\
		For the reverse inclusion, which is (\ref{eq:mult}), by  \cite[Proposition
		22.7]{SV}, we only need to show that if $\alpha$ is a cycle in $\CH^2(F)_{(0)}$,
		then $g\cdot \alpha$ belongs to $\CH^3(F)_{(0)}$.  
		To this end, we consider the correspondence 
		\[ \Gamma:= \pi_F^4\circ \Gamma_\iota\circ {}^t \Gamma_\iota\circ \pi_F^4\ \ \
		\in \CH^5(F\times F)\ ,\]
		where $\iota\colon H\hookrightarrow F$ denotes the inclusion of a hyperplane
		with respect to the Pl\"ucker embedding. Clearly, $\Gamma$ is homologically
		trivial. But $\Gamma$ is universally defined, and so Theorem \ref{thm:mainBD2}
		implies that $\Gamma$ is rationally trivial. The action of $\Gamma$ on
		$\CH^2(F)_{(0)}$ is the same as
		\[ \CH^2(F)_{(0)}\ \xrightarrow{\cdot g}\ \CH^3(F)\ \xrightarrow{}\
		\CH^3(F)_{(2)}\ \]
		(where the second arrow is projection on a direct summand),
		and so we are done.
	\end{proof}
	
	With notations as in \S \ref{sect:LLSvS},
	it seems that the final missing inclusion \eqref{eq:mult2} cannot be obtained
	from considering the subring 
	$\im\left(\CH^\ast(\rF\times_{B}\rF)\to \CH^\ast(F_{b}\times F_{b})\right)$.
	Rather, a streamlined proof of all inclusions $\CH^i(F)_{(j)} \cdot
	\CH^{i'}(F)_{(j')} \subseteq \CH^{i+i'}(F)_{(j+j')}$ would follow from
	establishing that the Chow--K\"unneth decomposition \eqref{eq:CK-F} is
	\emph{multiplicative} in the sense of \cite[\S 8]{SV}, meaning that 
	$$\pi_F^k \circ \delta_F\circ (\pi^i_F \otimes \pi_F^j) = 0 \quad \text{in }
	\CH^8(F\times F\times F), \quad \text{for all } k\neq i+j,$$
	where $\delta_F$ denotes the class of the small diagonal in $F\times F\times F$
	viewed as a correspondence from $F\times F$ to $F$.
	This in turn would follow from establishing the Franchetta property for the
	relative cube of the universal Fano variety of lines, \ie \ from showing that
	$$\im\left(\CH^\ast(\rF\times_{B}\rF\times_{B}\rF)\to \CH^\ast(F_{b}\times
	F_{b}\times F_{b})\right)$$ injects into cohomology by the cycle class map for
	all $b$. An approach would consist in first showing that this subring consists
	of ``tautological cycles''  and then in establishing enough ``tautological
	relations'', as was done in Propositions \ref{prop:AllTaut} and
	\ref{prop:InjTaut} in the case of the relative square.

	\bibliographystyle{amsplain}

\end{document}